\RequirePackage{fix-cm}
\documentclass[smallcondensed]{svjour3}     
\smartqed  
\usepackage{graphicx}
\usepackage[autostyle]{csquotes}
\usepackage{amsmath,mathtools}
\usepackage{amsfonts}
\usepackage{amssymb}
\usepackage{graphicx}
\usepackage{verbatim}
\usepackage{subcaption}
\usepackage[T1]{fontenc}
\usepackage{caption}
\usepackage{color}
\usepackage{bigints}
\usepackage[colorlinks,bookmarksopen,bookmarksnumbered,citecolor=red,urlcolor=red]{hyperref}
\usepackage[left=1.5in, top=1in, right=1.5in, bottom=1in]{geometry}
\usepackage{multirow}
\usepackage{hhline}

\DeclarePairedDelimiter{\norm}{\lVert}{\rVert}
\newtheorem{thm}{Theorem}[section]
\newtheorem{algorithm}[theorem]{Algorithm}

\newcommand{\be}{\begin{equation}}
\newcommand{\ee}{\end{equation}}
\newcommand{\bea}{\begin{eqnarray}}
\newcommand{\eea}{\end{eqnarray}}
\newcommand{\beas}{\begin{eqnarray*}}
\newcommand{\eeas}{\end{eqnarray*}}

\newcommand{\bfn}{\ensuremath{\mathbf{n}}}

\newcommand{\Reynolds}{\ensuremath{\mathrm{Re}}}
\newcommand{\LL}{\ensuremath{\mathrm{L}}}

\newcommand{\dt}{\ensuremath{\Delta t}}

\newcommand{\MTheta}{\ensuremath{\mathrm{\Theta}}}

\begin{document}

\title{Ensemble algorithm for parametrized flow problems with energy stable open boundary conditions
}

\titlerunning{Improved ensemble method}        

\author{Aziz Takhirov         \and
        Jiajia Waters 
}

\authorrunning{A. Takhirov \and J. Waters} 

\institute{A. Takhirov \at
              Department of Mathematical and Statistical Sciences, University of Alberta, Edmonton, AB, T6G 2G1, Canada.  \\
              \email{takhirov@ualberta.ca}                       
           \and
           J. Waters \at
              Los Alamos National Laboratory, Los Alamos, NM 87545, USA. \\
               \email{jwaters@lanl.gov} 
}


\maketitle

\begin{abstract}
We propose novel ensemble calculation methods for Navier-Stokes equations subject to various 
initial conditions, forcing terms and viscosity coefficients. 
We establish the stability of the schemes under a CFL condition involving velocity fluctuations. Similar to related works, the 
schemes require solution of a single system with multiple right-hand sides.
Moreover, we extend the ensemble calculation method to problems with open boundary conditions, with provable energy stability.
\keywords{Ensemble simulations \and open boundary conditions \and incompressible Navier-Stokes Equations}
\end{abstract}

\section{Introduction}
\label{intro}
We consider $J$ Navier-Stokes equations subject to perturbed initial conditions $u^0_j$, body forces $f_j$ and 
viscosity coefficients $\nu_j, j=\overline{1,J}$:
\begin{align}
\partial_t u_j + u_j \cdot \nabla u_j -  \nabla \cdot \left( \nu_j(x) \nabla u_j \right) + \nabla p_j & = f_j(x,t) \ \text{in} \ \Omega, \label{eq:NSE1} \\
\nabla \cdot u_j & = 0 \ \text{in} \ \Omega, \label{eq:NSE2} \\
u_j & = 0 \ \text{on} \ \Gamma_D, \label{eq:NSE3} \\
 \left( - \nu \nabla  u_j + p_j \mathrm{I} \right) \mathbf{n} & = \mathrm{S_\Gamma}\left(u_j,p_j \right) \ \text{on} \ \Gamma_N, \label{eq:NSE4} \\
 u_j(x,0) & = u^0_j(x) \ \text{in} \ \Omega, \label{eq:NSE5}
\end{align}
where $\Omega$ denotes the flow domain, and $\Gamma$ is its boundary. We assume that  $\Gamma$ is 
decomposed into non-overlapping Dirichlet $\Gamma_D$ and Neumann (open) $\Gamma_N$ boundaries. The choice of the stabilization term
$\mathrm{S_\Gamma}\left(u_j,p_j \right)$ will be discussed in the next section. \color{black}

When the system \eqref{eq:NSE1}-\eqref{eq:NSE4} is solved by linearly implicit methods, 
the corresponding linear system matrix will depend on the ensemble member $j$, due to the nonlinear and diffusion terms, and thus must be assembled $J$ times.
The advantage of the semi-implicit approach is obvious; one can pick the timestep solely based on accuracy considerations.
However, in practical applications, such as sensitivity analysis of the scheme to problem parameters
\cite{doi:10.1002/fld.1168}, reduced order modelling 
\cite{BURKARDT2006337,HOWARD2017333,WALTON20138930} and 
ensemble forecasting \cite{doi:10.1175/JAS-D-14-0250.1,doi:10.1175/1520-0493(1997)125<3297:EFANAT>2.0.CO;2}, $J$ tends to be quite large.
Solving one system of the form \eqref{eq:NSE1}-\eqref{eq:NSE4} is challenging by itself, and 
the computational cost of obtaining accurate solutions of all ensembles members by this approach maybe prohibitive.

In the $\nu_j = \nu$ case, one alternative to a semi-implicit approach is to treat the nonlinear term fully explicitly,
and thereby assemble a single linear system once for all. However, for high $\Reynolds$ flows, 
this strategy induces a very restrictive timestep condition, and the computational cost of the scheme 
could exceed the cost of the semi-implicit approach. The computational cost increases even further on adaptively refined meshes.

The first work for the efficient ensemble calculation was proposed in \cite{Jiang_2014}, which considered the $\nu_j = \nu = \mathrm{const}$ case.
The scheme was first order in time, and suitable for low $\Reynolds$ number flows. 
The idea was later extended to higher 
order schemes and high $\Reynolds$ flows in \cite{NUM:NUM21908}, \cite{Jiang2015}, and \cite{NUM:NUM22024}. The unifying idea in all of these
works is to split the advecting velocity in the nonlinear term into ensemble mean plus fluctuating part, make it explicit, and then treat the first 
nonlinear term semi-implicitly, while make the fluctuating part fully explicit.
The energy stability then can be shown to hold under a timestep restriction involving the velocity fluctuations, which should not be
as restrictive as the fully explicit approach.

The case of the multiple, variable viscosity coefficients has been recently addressed in \cite{doi:10.1093/imanum/dry029}. 
Since $-\nabla \cdot \left( \nu_j \nabla u_j \right)$ term is nonlinear with respect to ensemble member $j$, the splitting similar
to the treatment of the advection term was considered.
Denoting the mean viscosity by $\overline{\nu} = \frac{1}{J}\sum\limits_{j=1}^J \nu_j$, the following scheme was studied:
\begin{align}
\frac{u_j^{n+1} - u_j^{n}}{\Delta t} + \overline{u^n} \cdot \nabla u_j^{n+1} 
+ \left(u^n_j - \overline{u^n}\right) \cdot \nabla u_j^{n} & - \nabla \cdot \left( \nu_j \nabla u_j^n \right) \notag \\
- \nabla \cdot \left( \overline{\nu} \nabla \left(u^{n+1}_j-u^n_j \right)\right)  
+ \nabla p_j^{n+1} & = f_j(x,t_{n+1}) \ \text{in} \ \Omega, \label{eq:Nan1} \\
\nabla \cdot u_j^{n+1} & = 0 \ \text{in} \ \Omega \label{eq:Nan2}.
\end{align}
Although, the resulting linear system is independent of the ensemble member $j$, the stability of the scheme 
holds, besides a timestep restriction, under an additional assumption: 
\begin{equation}
\frac{|\nu_j - \overline{\nu}|}{\overline{\nu}} \le \sqrt{\mu}, \text{ for some } \mu \in \left[0,1\right).
\label{eq:NanViscosityCondition}
\end{equation}
The second order extension of \eqref{eq:Nan1}-\eqref{eq:Nan2} has been developed in \cite{Nan2ndOrderVisc}.

In this work, we consider a different treatment of the diffusive term. Denoting the maximum of the viscosities
as $\nu_\infty$, we consider the following scheme:
\begin{align}
\frac{u_j^{n+1} - u_j^{n}}{\Delta t} + \overline{u^n} \cdot \nabla u_j^{n+1} 
+ \left(u^n_j - \overline{u^n}\right) \cdot \nabla u_j^{n} & - \nabla \cdot \left( \nu_j \nabla u_j^n \right) \notag \\
- \nabla \cdot \left( \nu_\infty \nabla \left(u^{n+1}_j-u^n_j \right)\right)  
+ \nabla p_j^{n+1} & = f_j(x,t_{n+1}) \ \text{in} \ \Omega, \label{eq:StrongNSE1} \\
\nabla \cdot u_j^{n+1} & = 0 \ \text{in} \ \Omega \label{eq:StrongNSE2},
\end{align}
which allows to obtain a scheme with better stability properties. In particular, we are able to avoid any restriction on the viscosity 
coefficients in the case of first order scheme, cf. Theorem \ref{StabilityLemma}. The stability proof for the second order scheme 
assumes restrictions on viscosities, but the limitation is less stringent then the one proposed in \cite{Nan2ndOrderVisc}.

All previous works on ensembles were restricted to the homogeneous Dirichlet boundary conditions.
In this work, we also extend the ensemble scheme to problems with open boundaries, i.e., $\Gamma_N \neq \emptyset$. 
In this case, the stability requires additional time-step restriction on the part of the open boundary.

In this presentation, we restrict the analysis to the constant viscosity case for simplicity. One important example of the non-constant viscosity occurs when
the eddy viscosity hypothesis is applied to the ensemble of $\mathrm{Re} \gg 1$ flows. The schemes we propose can be easily extended 
to this case as well, when combined with the nonlinear filter based stabilization method of \cite{Takhirov2018}.

This paper is arranged as follows. Section 2 contains the notation and discussion on open boundary conditions.
Section 3 presents the weak formulations of the Algorithms. Section 4 proves energy stability, and Section 5 contains numerical experiments. 
The last Section gives a conclusion of the studies.


\section{Preliminaries}
\label{sec:Preliminaries}
\subsection{Notations}
Given ensemble $g_1,...,g_J$ of a quantity $g$, we define the fluctuation in $j-$th member as
\begin{equation*}
g^{'}_j = g_j - \overline{g},
\end{equation*}
and its maximum and minimum values by 
\begin{equation*}
g_\infty = \max\limits_{1\le j \le J}{g_j}, \text{ and } g_0 = \min\limits_{1\le j \le J}{g_j}, 
\end{equation*}
\color{black}respectively.
%
%
The $L^{2}(\Omega )$ norm and inner product will be denoted by $\Vert {\cdot 
}\Vert $ and $(\cdot ,\cdot )$, while the $L^\infty$ norm over a domain/surface $\gamma$ will be denoted by $\Vert {\cdot 
}\Vert _{\infty,\gamma}$.
For simplicity of the presentation, we assume no-slip boundary
condition on $\Gamma_D$. In this setting, the appropriate velocity and
pressure spaces  are defined as
\begin{equation*}
X:=(H_{0}^{1}(\Omega ))^{d},\  X_D: = \left\lbrace v \in \left(H^1 \left(\Omega \right)\right)^d: v = 0 \ \text{on } \Gamma_D \right \rbrace,
\ Q:=L_{0}^{2}(\Omega ).
\end{equation*}%
We use as the norm on $X$ and $X_D$, the seminorm $\Vert \nabla v\Vert _{L^{2}}$. The space of divergence free
functions is given by 
\begin{equation*}
V:=\{v\in X:\ (\nabla \cdot v,q)=0\ \ \forall q\in Q\}\,.
\end{equation*}

The dual spaces $X^* = H^{-1}(\Omega)$ and $X^*_D = H^{-1}_D(\Omega)$ are equipped with norms
\begin{equation*}
\|f\|_{-1} = \sup\limits_{v \in X} \frac{\langle f, v \rangle}{\|\nabla v\|}, \text{ and }
\|f\|_{-1,*} = \sup\limits_{v \in X_D} \frac{\langle f, v \rangle}{\|\nabla v\|},
\end{equation*}
where $\langle \cdot , \cdot \rangle $ refers to duality pairings.
We denote conforming velocity, pressure finite element spaces based on
an edge to edge triangulations (tetrahedralizations) of $\Omega$ (with maximum element diameter $h$) by
\begin{equation*}
X_{h}\subset X \left(X_{D,h}\subset X_D\right), Q_{h}\subset Q.
\end{equation*}
We assume that $X_{h},\ Q_{h}$\ satisfy the usual inf-sup stability condition \cite{AQ92}. The space of discrete, weakly divergence free functions is given by
\begin{equation*}
V_h:=\{v_h\in X_h:\ (\nabla \cdot v_h,q_h)=0,\ \ \forall q_h\in Q_h\}\,.
\end{equation*}
The trilinear term is denoted by
\begin{equation*}
 b(u,v,w) = (u \cdot \nabla v, w).
\end{equation*}
In discrete setting, $b(u,v,w)$ must be skew-symmetrized to ensure energy stability of the scheme. 
There are multiple variations discussed in the literature, cf. \cite{Charnyi2017289} for one recent result.
In our analysis and numerical tests, we will make use of the following skew-symmetrization:
\begin{align*}
b_1(u,v,w) := b(u,v,w) + \frac{1}{2}(\nabla \cdot u, w \cdot v)
\end{align*}
Integration by parts formula shows that 
\begin{align}
b_1(u,v,w) = \left(u \cdot \mathbf{n}, v \cdot w \right)_\Gamma - b_1(u,w,v), \label{eq:SkewSymmetry1} 
\end{align}
which in particular implies that 
\begin{align}
b_1(u,v,v) = \left( \frac{u \cdot \mathbf{n}}{2}, |v|^2 \right)_\Gamma. \label{eq:SkewSymmetry2} 
\end{align}
We will also make use of the Gronwall's Lemma.
\begin{lemma}\label{Gronwall}(\textbf{Gronwall's inequality}.)
Assume $\{a_n\},\{b_n\} $ are nonnegative sequences, $ c > 0 $ and
\begin{equation*}
a_n \le c + \sum\limits_{0 \le k < n}a_k b_k \ \text{for} \ n \ge 0.
\end{equation*}
Then
\begin{equation*}
a_n \le c \prod\limits_{0 \le k < n}(1+b_k) \le c \exp\left( \sum\limits_{0 \le k < n}b_k \right) \ \text{for} \ n \ge 0.
\end{equation*}
\end{lemma}
\subsection{Open boundary conditions}
Open boundary conditions \eqref{eq:NSE4} are often used to truncate a big physical domain to make the problem tractable,
or when the outflow profile can not be determined. Many choices of $\mathrm{S}_\Gamma$ have been proposed in the literature, and we 
refer the reader to \cite{doi:10.1002/fld.1650181006,FLD:FLD307,doi:10.1002/cnm.2918} for review of the topic for incompressible flows.
Our choice of $\mathrm{S}_\Gamma $ used in this work was proposed and benchmarked in \cite{Dong2015254,DONG2015300}. It belongs
to a family of velocity-penalization boundary conditions, and has been also successfully tested in physiological regimes as well \cite{doi:10.1002/cnm.2918}.
Moreover, among many other possible variations of the velocity-penalization boundary conditions, the approach we consider allows
for stable semi-implicit time-stepping (when $J=1$).
\color{black}

To discuss our choice of open boundary condition, let's partitition the $\Gamma_{N}$ boundary into outflow and backflow regions: 
$$\Gamma_{N} = \Gamma_{j,N}^{+} \cap \Gamma_{j,N}^{-}, \text{ where}$$
 $$\Gamma_{j,N}^{+} : = \left\lbrace x \in \Gamma_{N}: \left( u_j \cdot \bfn \right) (x) > 0 \right\rbrace
\text{ and } \Gamma_{j,N}^{-} : = \left\lbrace x \in \Gamma_{N}: \left( u_j \cdot \bfn \right) (x) \le 0 \right\rbrace.$$
We then take
\begin{align}
 \mathrm{S}_\Gamma := \frac{\left(u_j \cdot \bfn \right)u_j}{2} \left( \mathrm{H\left(u_j \cdot \bfn \right)} - 1\right) + \LL \partial_t u_j
 \label{eq:StableContinuousOBC} 
 \end{align}
where $\bfn $ denotes the unit normal on the boundary, $\mathrm{H}$ is the Heaviside function and $\LL$ is the characteristic length scale.
Up to the factor $\LL$, \eqref{eq:StableContinuousOBC} is same as the
the convective-like open boundary condition considered in \cite{DONG2015300}. Using $\LL$ instead of the original constant $\frac{\nu}{U}$ allows us to obtain a stability bound 
with a favourable constant.

For the implementation of the open boundary conditions, we introduce another trilinear term:
\begin{align*}
b_2(u,v,w) := -\frac{1}{2}((u \cdot n) \MTheta_0(u \cdot \bfn), v \cdot w)_{\Gamma_N},
\end{align*}
where $$\MTheta_0(u \cdot \bfn) = \frac{1}{2} \left( 1 - \mathrm{tanh}\frac{ u \cdot \bfn}{\varepsilon U_0}\right) \simeq \mathrm{H}(u \cdot \bfn) - 1, 
\text{ and } \MTheta_1(u \cdot \bfn) = 1 - \MTheta_0(u \cdot \bfn),$$ 
$\varepsilon \ll 1$ and $U_0$ is a reference speed.
We note that for the problems with open boundaries, $b_1(\cdot, \cdot, \cdot)$ is more accurate 
than another commonly used skew-symmetrization $$b_3(u,v,w) := \frac{1}{2}\left( b(u,v,w) - b(u, w, v) \right),
$$ in a sense that, if 
$u \in X_D$, divergence free and $v \in X_D$, then 
\begin{equation}
b(u,v,v) = b_1(u,v,v) = \bigintsss\limits_{\Gamma_N} \frac{u \cdot \bfn }{2} |v|^2, 
\text{ while } b_3(u,v,v) = 0. \label{eq:NL_ener}
\end{equation}
\section{Numerical schemes}
\subsection{First order schemes}
\label{sec:WeakAlghtms1}
For the case of the pure Dirichlet boundary condition, our first order 
algorithm approximating \eqref{eq:NSE1}-\eqref{eq:NSE4} takes the following form.
\begin{algorithm}\label{thealgorithm1WK}
Given $J$ initial velocities $u^0_j \in V$, forcing terms $f_j \in H^{-1}(\Omega)$ and viscosities $\nu_j$, time step $\Delta t>0$,
find $(u_{j,h}^{n+1},p_{j,h}^{n+1})\in (X_h,Q_h)$, $n=0,1,...,N-1,$ satisfying
\begin{align}
 \frac{\left(u_{j,h}^{n+1}-u_{j,h}^n,v_h\right)}{\Delta t}
 + b_1\left(\overline{u_h^n}, u_{j,h}^{n+1},v_h\right) + b_1\left(u_{j,h}^{n'}, u_{j,h}^{n},v_h\right)  &- 
  \left(p_{j,h}^{n+1},\nabla \cdot v_h\right)  \notag \\
+ \nu_j \left(\nabla u_{j,h}^{n},\nabla v_h\right) + \nu_\infty \left( \nabla (u_{j,h}^{n+1} - u_{j,h}^{n}), \nabla v_h \right)
 & =  \langle f_j(t^{n+1}),v_h \rangle,  \label{eq:Algth1aWK}\\
 \left(\nabla \cdot u_{j,h}^{n+1},q_h\right)  &=  0, \label{eq:Algth1bWK}
\end{align}
for all $v_h\in X_h$ and $q_h\in Q_h$.
\end{algorithm}
Now we turn to the problems with open boundaries. To this end, we first derive the weak formulation of the continuous system
\eqref{eq:NSE1}-\eqref{eq:NSE5} under the following perturbation of the open boundary condition \eqref{eq:StableContinuousOBC} for the
ensemble case:
\begin{equation}
 \left( - \nu_j \nabla u_{j} - \dt \nu_\infty \nabla \partial_t u_{j} + p_{j} \mathrm{I} \right) \bfn = 
  \frac{\left(u_j \cdot \bfn \right)u_j}{2} \mathrm{\Theta_0}\left(u_j \cdot \bfn \right) + \LL \partial_t u_j.
 \label{eq:StableContinuousOBC1}
\end{equation}
Then the weak form takes following form:
\begin{align}
 \left(\partial_t u_j,v \right) + \mathrm{L} \left(\partial_t u_j,v \right)_{\Gamma_N}
 + b_1\left(u_j, u_j,v\right) + b_2\left(u_j, u_j, v \right) & + \nu_j \left(\nabla u_{j},\nabla v \right) \notag \\
  + \dt \nu_\infty \left( \nabla \partial_t u_{j}, \nabla v \right) 
 - \left(p_{j} ,\nabla \cdot v \right) & = \langle f_j ,v \rangle,  \label{eq:Algth2aCWK}  \\
 \left(\nabla \cdot u_{j},q_h \right) & = 0, \label{eq:Algth2bCWK}
\end{align}
To derive the scheme for the ensemble calculation, we treat both nonlinear terms and the viscous term as in Algorithm \ref{thealgorithm1WK}:
\begin{algorithm}\label{thealgorithm2WK}
Given $J$ initial velocities $u^0_j \in V$, forcing terms $f_j \in X_D^*$ and viscosities $\nu_j$, time step $\Delta t>0$,
find $(u_{j,h}^{n+1},p_{j,h}^{n+1})\in (X_{D,h},Q_h)$, $n=0,1,...,N-1$ satisfying
\begin{align}
 \frac{\left(u_{j,h}^{n+1}-u_{j,h}^n,v_h \right) + \mathrm{L} \left(u_{j,h}^{n+1}-u_{j,h}^n,v_h \right)_{\Gamma_N}}{\Delta t} 
& + \sum\limits_{i=1}^2 \left( b_i\left(\overline{u_h^n}, u_{j,h}^{n+1},v_h\right) \right.  \notag \\ 
 +  \left. b_i\left(u_{j,h}^{n'}, u_{j,h}^{n},v_h\right) \right)
 & + \nu_j \left(\nabla u_{j,h}^{n},\nabla v_h\right) \label{eq:Algth2aWK} \\ 
 + \nu_\infty \left( \nabla (u_{j,h}^{n+1} - u_{j,h}^{n}), \nabla v_h \right) - \left(p_{j,h}^{n+1},\nabla \cdot v_h\right) & = 
 \left\langle f_j^{n+1},v_h \right\rangle \notag \\
 \left(\nabla \cdot u_{j,h}^{n+1},q_h \right) & = 0, \label{eq:Algth2bWK}
\end{align}
for all $v_h\in X_{D,h}$ and $q_h\in Q_h$.
\end{algorithm}
Another scheme can be derived via the following substitution 
$$\sum\limits_{i=1}^2 b_i\left(u_{j,h}^{n'}, u_{j,h}^{n},v_h\right) \leftarrow
b_3\left(u_{j,h}^{n'}, u_{j,h}^{n},v_h\right): $$
\begin{algorithm}\label{thealgorithm3WK}
Given $J$ initial velocities $u^0_j \in V$, forcing terms $f_j \in X_D^*$ and viscosities $\nu_j$, time step $\Delta t>0$,
find $(u_{j,h}^{n+1},p_{j,h}^{n+1})\in (X_{D,h},Q_h)$, $n=0,1,...,N-1$ satisfying
\begin{align}
 \frac{\left(u_{j,h}^{n+1}-u_{j,h}^n,v_h \right)}{\Delta t} 
 + \sum \limits_{i=1}^2 b_i\left(\overline{u_h^n}, u_{j,h}^{n+1},v_h\right) + b_3\left(u_{j,h}^{n'}, u_{j,h}^{n},v_h \right) 
& + \nu_j \left(\nabla u_{j,h}^{n},\nabla v_h\right) \notag \\
 + \nu_\infty \left( \nabla (u_{j,h}^{n+1} - u_{j,h}^{n}), \nabla v_h \right) 
  - \left(p_{j,h}^{n+1},\nabla \cdot v_h\right) &= \left\langle f_j^{n+1},v_h \right\rangle,  \label{eq:Algth3aWK}  \\
 \left(\nabla \cdot u_{j,h}^{n+1},q_h \right) &= 0, \label{eq:Algth3bWK}
\end{align}
for all $v_h\in X_{D,h}$ and $q_h\in Q_h$.
\end{algorithm}
Note that we set $\mathrm{L} = 0$ in this case, as this term is not necessary for proving the stability in this case.
\subsection{Second order schemes}
One way of obtaining a second order extension of the Algorithms \ref{thealgorithm1WK}-\ref{thealgorithm3WK} is to use BDF2 approximation of
the time derivative, and second order extrapolation 
$E_{j,h}^n : = 2u_{j,h}^n - u_{j,h}^{n-1}$ instead of $u_{j,h}^n$. However, we were only able to prove stability of the schemes 
in this case under 
very stringent assumptions on $\frac{\nu_\infty}{\nu_0}$. In order to have a scheme with better properties, we take an inspiration from 
\cite{JIANG2016388} and consider the following schemes where we also add a stabilization in the implicit nonlinear terms: 
\begin{algorithm}\label{thealgorithm4WK}
Given $J$ initial velocities $u^0_j \in V$, forcing terms $f_j \in H^{-1}(\Omega)$ and viscosities $\nu_j$, time step $\Delta t>0$, 
a constant $\gamma \in [0, 2)$, find $(u_{j,h}^{n+1},p_{j,h}^{n+1})\in (X_h,Q_h)$, $n=0,1,...,N-1,$ satisfying
\begin{align}
 \frac{\left(3 u_{j,h}^{n+1}- 4 u_{j,h}^n + u_{j,h}^{n-1},v_h \right)}{2\Delta t}
& + b_1\left(\overline{E_h^n}, u_{j,h}^{n+1} + \gamma \left(u_{j,h}^{n+1} - E_{j,h}^{n} \right),v_h \right) \notag \\
 + b_1\left(E_{j,h}^{n'}, E_{j,h}^{n},v_h \right) & - \left(p_{j,h}^{n+1},\nabla \cdot v_h \right) + 
 \nu_j \left(\nabla E_{j,h}^{n},\nabla v_h \right) \label{eq:Algth4aWK} \\
 & + \nu_\infty \left( \nabla \left(u_{j,h}^{n+1} - E_{j,h}^{n} \right), \nabla v_h \right)
  =  \left\langle f_j(t^{n+1}),v_h \right\rangle, \notag  \\
 \left(\nabla \cdot u_{j,h}^{n+1},q_h \right) & =  0, \label{eq:Algth4bWK}
\end{align}
for all $v_h\in X_h$ and $q_h\in Q_h$.
\end{algorithm}
\begin{algorithm}\label{thealgorithm5WK}
Given $J$ initial velocities $u^0_j \in V$, forcing terms $f_j \in X_D^*$ and viscosities $\nu_j$, time step $\Delta t>0$,
a constant $\gamma \in [0, 2)$,
find $(u_{j,h}^{n+1},p_{j,h}^{n+1})\in (X_{D,h},Q_h)$, $n=0,1,...,N-1$ satisfying
\begin{align}
 \frac{\left(3 u_{j,h}^{n+1}- 4 u_{j,h}^n + u_{j,h}^{n-1},v_h \right) + 
 \mathrm{L} \left(3 u_{j,h}^{n+1}- 4 u_{j,h}^n + u_{j,h}^{n-1},v_h \right)_{\Gamma_N}}{2\Delta t} & \notag \\
  + \sum\limits_{i=1}^2 \left( b_i\left(\overline{E_h^n}, u_{j,h}^{n+1} + \gamma \left(u_{j,h}^{n+1} - E_{j,h}^{n} \right),v_h\right) \right. 
 + \left. b_i\left(E_{j,h}^{n'}, E_{j,h}^{n},v_h\right) \right) &  \label{eq:Algth5aWK} \\
 + \nu_j \left(\nabla E_{j,h}^{n},\nabla v_h\right) + \nu_\infty \left( \nabla (u_{j,h}^{n+1} - E_{j,h}^{n}), \nabla v_h \right) 
  - \left(p_{j,h}^{n+1},\nabla \cdot v_h\right) & = \left\langle f_j^{n+1},v_h \right\rangle,  \notag  \\
 \left(\nabla \cdot u_{j,h}^{n+1},q_h \right) & = 0, \label{eq:Algth5bWK}
\end{align}
for all $v_h\in X_{D,h}$ and $q_h\in Q_h$.
\end{algorithm}
\begin{algorithm}\label{thealgorithm6WK}
Given $J$ initial velocities $u^0_j \in V$, forcing terms $f_j \in  X_D^*$ and viscosities $\nu_j$, time step $\Delta t>0$,
find $(u_{j,h}^{n+1},p_{j,h}^{n+1})\in (X_{D,h},Q_h)$, $n=0,1,...,N-1$ satisfying
\begin{align}
 \frac{\left(3 u_{j,h}^{n+1}- 4 u_{j,h}^n + u_{j,h}^{n-1},v_h \right)}{2\Delta t}
 + \sum \limits_{i=1}^2 b_i\left(\overline{E_h^n}, u_{j,h}^{n+1} \right. & + \left. \gamma \left(u_{j,h}^{n+1} - E_{j,h}^{n} \right) ,v_h\right) \notag \\
 + b_3\left(E_{j,h}^{n'}, E_{j,h}^{n},v_h \right) + \nu_j \left(\nabla E_{j,h}^{n},\nabla v_h\right) & + 
 \nu_\infty \left( \nabla (u_{j,h}^{n+1} - E_{j,h}^{n}), \nabla v_h \right) \notag \\ 
  - \left(p_{j,h}^{n+1},\nabla \cdot v_h\right) & = \left\langle f_j^{n+1},v_h \right\rangle,  \label{eq:Algth6aWK}  \\
 \left(\nabla \cdot u_{j,h}^{n+1},q_h \right) & = 0, \label{eq:Algth6bWK}
\end{align}
for all $v_h\in X_{D,h}$ and $q_h\in Q_h$.
\end{algorithm}
The optimal value of $\gamma$ depends on the ratio of viscosities, and is chosen following the procedure outlined in the
Subsection \ref{sec:Stab_2nd}. 

\color{black}
All the Algorithms \ref{thealgorithm1WK}-\ref{thealgorithm6WK} give rise to a matrices that are independent of the ensemble member,
and thus require that only a single coefficient matrix is stored along with $J$ right-hand sides at each time step.
The resulting linear systems could be solved efficiently using solvers for systems with multiple right-hand sides,
cf. \cite{JBILOU199997,Heyouni2005}.
\section{Theoretical resutls}
\label{sec:theory}
\subsection{First order schemes}
\subsubsection{Stability with homogeneous Dirichlet boundary conditions}
\label{sec:StabDirichlet1stOrder}
In this subsection, we consider the $\Gamma_N = \emptyset$ case, and we will establish the stability under a time step condition:
\begin{eqnarray}
\frac{\Delta t}{\nu_j} \left( \norm[\big]{u_{j,h}^{n'} }_\infty + \frac{\mathrm{diam}\left( \Omega \right)}{d} 
\norm[\big]{\nabla \cdot u_{j,h}^{n '} }_\infty \right)^2 \le 1, \ \forall j = 1,...,J.
\label{eq:StabCondition}
\end{eqnarray}
\begin{thm}\label{StabilityLemma}
Let $$ \mathrm{Ener}^{n}_j: = \frac{\norm[\big]{u_{j,h}^{n} }^2}{2} + \Delta t \frac{\nu_\infty}{2} \norm[\big]{\nabla u_{j,h}^{n} }^2. $$
If \eqref{eq:StabCondition} holds for each time step $n \ge 1$, then the solutions to Algorithm \ref{thealgorithm1WK} satisfy
\begin{align}
 \mathrm{Ener}^{N}_j \le \mathrm{Ener}^0_j + \sum\limits_{n=0}^{N-1} \frac{\Delta t}{2 \nu_j} \norm[\big]{ f_j(t^{n+1})}^2_{-1} \ \text{ for all } j.
\label{eq:Stability}
\end{align}
\end{thm}
\begin{proof}
Choose $v_h = u_{j,h}^{n+1}$ in \eqref{eq:Algth1aWK}, $q_h = p_{j,h}^{n+1}$ in \eqref{eq:Algth1bWK} and add them to get 
\begin{align}
\frac{\left( u_{j,h}^{n+1}-u_{j,h}^n,u_{j,h}^{n+1} \right)}{\Delta t} + b_1 \left(u_{j,h}^{n'} , u_{j,h}^n,u_{j,h}^{n+1} \right) 
& + \nu_\infty \norm[\big]{\nabla u_{j,h}^{n+1} }^2 \label{eq:Stab1} \\
+ \left(\nu_j - \nu_\infty \right)\left(\nabla u_{j,h}^{n}, \nabla u_{j,h}^{n+1} \right) & = \langle f_j(t^{n+1}),u_{j,h}^{n+1} \rangle.  \notag
\end{align}
Applying the polarization identity gives
\begin{align}
\frac{\norm[\big]{u_{j,h}^{n+1} }^2 - \norm[\big]{u_{j,h}^{n} }^2 + \norm[\big]{ u_{j,h}^{n+1} - u_{j,h}^{n} }^2}{2 \Delta t} 
+ \nu_\infty \norm[\big]{ \nabla u_{j,h}^{n+1} }^2 
& = \langle f_j(t^{n+1}), u^{n+1}_{j,h}  \rangle
 \label{eq:Stab2} \\
- b_1\left(u_{j,h}^{n'}, u_{j,h}^n,u_{j,h}^{n+1}\right) + \left(\nu_\infty \right. & - \left. \nu_j \right) \left( \nabla u^n_{j,h}, \nabla u^{n+1}_{j,h} \right).
\notag
\end{align}
It remains to bound the terms on the right hand side.
Using the generalized H\"{o}lder's and Young's inequalities, we obtain
\begin{align}
b_1\left(u_{j,h}^{n'}, u_{j,h}^n,u_{j,h}^{n+1}\right) & = b_1\left(u_{j,h}^{n'}, u_{j,h}^{n},u_{j,h}^{n+1}-u_{j,h}^n\right) \label{eq:Stab3} \\
& \le \left( \norm[\big]{ u_{j,h}^{n '} }_\infty \norm[\big]{ \nabla u_{j,h}^{n}} + 
\frac{\norm[\big]{ \nabla \cdot u_{j,h}^{n '} }_\infty}{2} \norm[\big]{u_{j,h}^{n}} \right) \norm[\big]{u_{j,h}^{n+1} - u_{j,h}^n }  \notag \\
& \le \frac{ \norm[\big]{u_{j,h}^{n+1} - u_{j,h}^n }^2}{2 \Delta t} \notag \\
& + \frac{\Delta t}{2} \norm[\big]{ \nabla u_{j,h}^{n} }^2 \left( \norm[\big]{ u_{j,h}^{n '} }_\infty + 
\frac{\norm[\big]{ \nabla \cdot u_{j,h}^{n '} }_\infty}{2} C_p \right)^2, \notag
\end{align}
where $C_p = \frac{2 \mathrm{diam} \left(\Omega \right)}{d}$ is the Poincar\'{e}'s constant \cite[pg. 22]{Quar09}.
Further, we get 
\begin{align}
 \left(\nu_\infty - \nu_j \right) \left( \nabla u^n_{j,h}, \nabla u^{n+1}_{j,h} \right) \le
 \frac{\nu_\infty - \nu_j }{2} \norm[\big]{ \nabla u^n_{j,h} }^2 + 
\frac{\nu_\infty - \nu_j }{2} \norm[\big]{ \nabla u^{n+1}_{j,h} }^2, \label{eq:Stab4}
\end{align}
and 
\begin{align}
\langle f_j(t^{n+1}), u^{n+1}_{j,h} \rangle \le \frac{1}{2 \nu_j} \norm[\big]{ f_j(t^{n+1}) }^2_{-1} +
\frac{\nu_j}{2} \norm[\big]{ \nabla u^{n+1}_{j,h} }^2.
\label{eq:Stab5}
\end{align}
Combining \eqref{eq:Stab2}  - \eqref{eq:Stab5} yields
\begin{align}
\frac{\norm[\big]{u_{j,h}^{n+1} }^2 - \norm[\big]{u_{j,h}^{n} }^2}{2 \Delta t} & + \frac{\nu_\infty}{2} \left[ \norm[\big]{\nabla u_{j,h}^{n+1}}^2
- \norm[\big]{\nabla u_{j,h}^{n}}^2 \right] \label{eq:Stab6} \\
& + \left[\frac{\nu_j}{2} -  \frac{\Delta t}{2} \left( \norm[\big]{u_{j,h}^{n '}}_\infty + \frac{\norm[\big]{\nabla \cdot u_{j,h}^{n '}}_\infty}{2} C_p \right)^2 \right] 
\norm[\big]{\nabla u_{j,h}^{n}}^2 \notag \\
& \le \frac{1}{2 \nu_j} \norm[\big]{f_j(t^{n+1})}^2_{-1}.\notag
\end{align}
Under the CFL condition \eqref{eq:StabCondition}, the last term on the left hand side of \eqref{eq:Stab6} is nonnegative and summing over the 
timesteps completes the proof.
\end{proof}
%
\begin{remark}
We can obtain an improved stability bound 
\begin{align}
 \mathrm{Ener}^{N}_j + \nu_j \Delta t \sum \limits_{n=0}^{N-1} \norm[\big]{\nabla u^n_{j,h}}^2 \le C \left( \mathrm{Ener}^0_j + \sum\limits_{n=0}^{N-1} \frac{\Delta t}{2 \nu_j}  \norm[\big]{f_j(t^{n+1})}^2_{-1} \right).
\label{eq:StabilityDissipative}
\end{align}
If we assume the following, slightly restrictive timestep condition 
\begin{eqnarray}
\frac{2 \Delta t}{\nu_j} \left( \norm[\big]{u_{j,h}^{n'}}_\infty + \frac{\mathrm{diam}\left( \Omega \right)}{d} \norm[\big]{\nabla \cdot u_{j,h}^{n '}}_\infty \right)^2 \le 1, \ \forall j = 1,...,J.
\label{eq:StabConditionDissipative}
\end{eqnarray}
\end{remark}
\begin{proof}
 Consider the expression in the second line of \eqref{eq:Stab6}. Under \eqref{eq:StabConditionDissipative}, we obtain that 
\begin{align*}
\left[\frac{\nu_j}{2} -  \frac{\Delta t}{2} \left( \norm[\big]{u_{j,h}^{n '}}_\infty + \frac{\norm[\big]{ \cdot u_{j,h}^{n '}}_\infty}{2} C_p \right)^2 \right] 
\norm[\big]{\nabla u_{j,h}^{n}}^2 \ge \nu_j \norm[\big]{\nabla u_{j,h}^{n}}^2.
\end{align*}
Summing over the timesteps completes the proof.
\end{proof}
\subsubsection{Stability with outflow boundary conditions}
\label{sec:StabilityOutflow1stOrder}
Now we consider the case of $\Gamma_N \neq \emptyset$.
The stability for the Algorithm \ref{thealgorithm2WK} holds under the following two timestep conditions
\begin{eqnarray}
\frac{\Delta t}{\nu_j} \left( \norm[\big]{u_{j,h}^{n'}}_\infty + \frac{1}{\sqrt{\lambda_1}} \norm[\big]{\nabla \cdot u_{j,h}^{n '}}_\infty \right)^2 & \le 1, \ \forall j = 1,...,J,
\label{eq:StabConditionOBC1} \\
\frac{\dt}{8\nu_j} \norm[\Big]{ \left(u_{j,h}^{n'} \cdot \bfn \right) \MTheta_0 \left(u_{j,h}^{n'} \cdot \bfn \right)}^2_{\infty,\Gamma_N} & \le 1,  \ \forall j = 1,...,J,
\label{eq:StabConditionOBC2}
\end{eqnarray}
 where $ \lambda_1 > 0$ is the smallest eigenvalue of the mixed Dirichet-Neummann spectral problem 
 \begin{align}
 -\Delta u & = \lambda u \text{ in } \Omega, \notag \\
         u & = 0 \text{ on } \Gamma_D, \label{eq:eigenvalue} \\
         \frac{\partial u}{\partial \bfn} & = 0 \text{ on } \Gamma_N. \notag 
 \end{align}
\begin{thm}\label{StabilityLemmaOBC}
Let $$ \mathrm{Ener}^{n}_j: = \frac{\norm[\big]{u_{j,h}^{n}}^2 + \mathrm{L} \norm[\big]{u_{j,h}^{n}}^2_{\Gamma_N}}{2} + \Delta t \frac{\nu_\infty}{2} \norm{\nabla u_{j,h}^{n}}^2. $$
If for each time step $n \ge 1$, the conditions \eqref{eq:StabConditionOBC1}-\eqref{eq:StabConditionOBC2} hold,
then the solutions to Algorithm \ref{thealgorithm2WK} satisfy
\begin{align}
\mathrm{Ener}^{N}_j & + \Delta t \sum\limits_{n=0}^{N-1} \bigintsss\limits_{\Gamma_N} \left[ \frac{\overline{u^n_h} \cdot \bfn}{2} \left|u_{j,h}^{n+1}\right|^2 
\MTheta_1\left(\overline{u^n_h} \cdot \bfn\right)  + 
\frac{u^{n'}_{j,h} \cdot \bfn}{2} \left|u_{j,h}^{n}\right|^2 \MTheta_1\left(u^{n'}_{j,h} \cdot \bfn \right)  \right] \notag \\
& \le \exp\left( \frac{\nu_j T}{\mathrm{L}} \right) \left(
 \mathrm{Ener}^0_j + \sum\limits_{n=0}^{N-1} \frac{\Delta t}{2 \nu_j}  \norm[\big]{f_j(t^{n+1})}^2_{*,D}\right).
\label{eq:StabilityOBC}
\end{align}
\end{thm}
\begin{proof}
Choose the test functions $v_h = u_{j,h}^{n+1}$, $q_h = p^{n+1}_{j,h}$, and add the equations \eqref{eq:Algth2aWK}-\eqref{eq:Algth2bWK}. 
Taking into account \eqref{eq:NL_ener}, \color{black} the first nonlinear term becomes
\begin{align}
b_1 \left(\overline{u_h^n}, u_{j,h}^{n+1},u_{j,h}^{n+1}\right) 
& + b_1 \left( u_{j,h}^{n'}, u_{j,h}^{n},u_{j,h}^{n+1}\right) \notag \\ 
& = b_1 \left(\overline{u_h^n}, u_{j,h}^{n+1},u_{j,h}^{n+1}\right) + b_1 \left( u_{j,h}^{n'}, u_{j,h}^{n},u_{j,h}^{n}\right) \notag \\
& + b_1 \left( u_{j,h}^{n'}, u_{j,h}^{n},u_{j,h}^{n+1} - u_{j,h}^{n}\right) \notag \\
& = \bigintsss\limits_{\Gamma_N} \left[ \frac{ \overline{u^n_h} \cdot \bfn}{2} \left|u_{j,h}^{n+1}\right|^2 + 
\frac{u^{n'}_{j,h} \cdot \bfn}{2} \left|u_{j,h}^{n}\right|^2 \right] \notag \\
& + b_1 \left( u_{j,h}^{n'}, u_{j,h}^{n},u_{j,h}^{n+1} - u_{j,h}^{n}\right), \label{eq:Stab1OBC}
\end{align}
and similarly, 
\begin{align}
& b_2 \left(\overline{u_h^n}, u_{j,h}^{n+1},u_{j,h}^{n+1}\right) 
 + b_2 \left( u_{j,h}^{n'}, u_{j,h}^{n},u_{j,h}^{n+1}\right) \notag \\ 
& = - \bigintsss\limits_{\Gamma_N} \left[ \frac{ \overline{u^n_h} \cdot \bfn}{2} \left|u_{j,h}^{n+1}\right|^2 \MTheta_0(\overline{u^n_h} \cdot \bfn)  + 
\frac{u^{n'}_{j,h} \cdot \bfn}{2} \left|u_{j,h}^{n}\right|^2 \MTheta_0(u^{n'}_{j,h} \cdot \bfn)  \right] \notag \\
& + b_2 \left( u_{j,h}^{n'}, u_{j,h}^{n},u_{j,h}^{n+1} - u_{j,h}^{n}\right). \label{eq:Stab2OBC}
\end{align}
Taking \eqref{eq:Stab1OBC}-\eqref{eq:Stab2OBC} into account gives
\begin{align}
& \frac{\norm[\big]{u_{j,h}^{n+1} }^2 - \norm[\big]{u_{j,h}^{n} }^2 + \norm[\big]{u_{j,h}^{n+1} - u_{j,h}^{n}}^2}{2 \Delta t} + 
\nu_\infty \norm[\big]{\nabla u_{j,h}^{n+1}}^2 \notag \\
& + \mathrm{L} \frac{\norm[\big]{u_{j,h}^{n+1} }^2_{\Gamma_N} - \norm[\big]{u_{j,h}^{n} }^2_{\Gamma_N} + \norm[\big]{u_{j,h}^{n+1} - u_{j,h}^{n}}^2_{\Gamma_N}}{2 \Delta t}  
\notag \\
& + \underbrace{ \bigintsss\limits_{\Gamma_N} \left[ \frac{ \overline{u^n_h} \cdot \bfn }{2} \left|u_{j,h}^{n+1}\right|^2 \MTheta_1\left(\overline{u^n_h} \cdot \bfn \right) + 
\frac{u^{n'}_{j,h} \cdot \bfn}{2} \left|u_{j,h}^{n}\right|^2 \MTheta_1\left(u^{n'}_{j,h} \cdot \bfn \right)  \right]}_{:=F_{n+1} \ge 0} \label{eq:Stab3OBC} \\
& = \left\langle f_j(t^{n+1}), u^{n+1}_{j,h}  \right\rangle + \left(\nu_\infty - \nu_j \right) \left( \nabla u^n_{j,h}, \nabla u^{n+1}_{j,h} \right) \notag \\
& - \sum \limits_{i=1}^2 b_i \left( u_{j,h}^{n'}, u_{j,h}^{n},u_{j,h}^{n+1} - u_{j,h}^{n} \right).
\notag
\end{align}
The first two terms on the right hand side and $b_1 \left( u_{j,h}^{n'}, u_{j,h}^{n},u_{j,h}^{n+1} - u_{j,h}^{n}\right)$ are treated as
in the proof of Theorem \ref{StabilityLemma}, with $\lambda_1^{-\frac{1}{2}}$ now acting as a Poincare's constant. 
\color{black} As for the $b_2 \left( u_{j,h}^{n'}, u_{j,h}^{n},u_{j,h}^{n+1} - u_{j,h}^{n}\right)$, we apply Cauchy-Schwarz to get
\begin{align}
b_2 \left( u_{j,h}^{n'}, u_{j,h}^{n},u_{j,h}^{n+1} - u_{j,h}^{n}\right) & = 
 \bigintsss\limits_{\Gamma_N} \MTheta_0(u^{n'}_{j,h} \cdot \bfn) \frac{- u^{n'}_{j,h} \cdot \bfn }{2} u_{j,h}^{n} \cdot \left(u_{j,h}^{n+1} - u_{j,h}^{n} \right) 
 \notag \\
& \le \mathrm{L}\frac{\norm[\big]{u_{j,h}^{n+1} - u_{j,h}^{n}}^2_{\Gamma_N}}{2 \Delta t} \notag \\
& + \frac{\dt \norm[\Big]{ \left(u^{n'}_{j,h} \cdot \bfn \right) \MTheta_0\left(u^{n'}_{j,h} \cdot \bfn \right)}^2_{\infty,\Gamma_N}}{8 \mathrm{L}} 
\norm[\big]{u_{j,h}^{n}}^2_{\Gamma_N} \notag \\
& \le \mathrm{L}\frac{\norm[\big]{u_{j,h}^{n+1} - u_{j,h}^{n}}^2_{\Gamma_N}}{2 \Delta t} + \frac{\nu_j}{\mathrm{L}} \norm[\big]{u_{j,h}^{n}}^2_{\Gamma_N}. \label{eq:Stab4OBC} 
\end{align}
The last bounded has been obtained under \eqref{eq:StabConditionOBC2}. Putting everything together and summing over the timesteps yields
\begin{align}
\mathrm{Ener}^N_j + \dt \sum \limits_{n=1}^{N-1} F_{n+1} \le \mathrm{Ener}^0_j + 
\frac{ \nu_j}{\mathrm{L}} \dt \sum \limits_{n=1}^{N-1}\norm[\big]{u_{j,h}^{n}}^2_{\Gamma_N} 
+ \frac{\dt}{2 \nu_j} \sum \limits_{n=1}^{N-1}\norm[\big]{f_{j}^{n+1}}^2_{*,D}. 
\label{eq:Stab5OBC}
\end{align}
Gronwall's inequality completes the proof.
\end{proof}
Now we prove the stability of the Algorithm \ref{thealgorithm3WK} under 
\begin{eqnarray}
\frac{C \Delta t}{h \nu_j} \norm[\big]{\nabla u_{j,h}^{n '}}^2 \le 1, \ \forall j = 1,...,J, \ C = \mathcal{O}(1).
\label{eq:StabConditionOBC3}
\end{eqnarray}
\begin{thm}\label{StabilityLemmaOBC3}
Let $$ \mathrm{Ener}^{n}_j: = \frac{\norm[\big]{u_{j,h}^{n} }^2}{2} + \Delta t \frac{\nu_\infty}{2} \norm{\nabla u_{j,h}^{n}}^2.$$
If for each time step $n \ge 1$, the condition \eqref{eq:StabConditionOBC3} holds, then the solutions to Algorithm \ref{thealgorithm3WK} satisfy
\begin{align}
 \mathrm{Ener}^{N}_j & + 
 \Delta t \sum\limits_{n=0}^{N-1} \bigintsss\limits_{\Gamma_N} \frac{\overline{u^n_h} \cdot \bfn}{2} \left|u_{j,h}^{n+1}\right|^2 \MTheta_1 \left(\overline{u^n_h} \cdot \bfn \right)\notag \\ 
 & \le \mathrm{Ener}^0_j + \sum\limits_{n=0}^{N-1} \frac{\Delta t}{2 \nu_j}  \norm[\big]{f_j(t^{n+1})}^2_{*,D}.
\label{eq:StabilityOBC3_1stOrder}
\end{align}
\end{thm}
\begin{proof}
The proof is very similar to that of Theorems \ref{StabilityLemma}-\ref{StabilityLemmaOBC}. The only difference 
is bounding the $b_3 \left( \cdot, \cdot, \cdot \right)$
term:
\begin{align}
b_3\left(u_{j,h}^{n'}, u_{j,h}^n,u_{j,h}^{n+1}\right) & =
b_3\left(u_{j,h}^{n'}, u_{j,h}^{n},u_{j,h}^{n+1}-u_{j,h}^n\right) \label{eq:Stab2OBC2_1stOrder} \\
& \le C \norm[\Big]{\nabla u_{j,h}^{n'}} \norm[\Big]{\nabla u_{j,h}^{n}} 
\sqrt{ \norm[\Big]{ \nabla \left( u_{j,h}^{n+1} - u_{j,h}^n  \right) }  \norm[\Big]{u_{j,h}^{n+1} - u_{j,h}^n }}  \notag \\
& \le C h^{-1/2} \norm[\big]{\nabla u_{j,h}^{n'}} \norm[\big]{\nabla u_{j,h}^{n}} \norm[\big]{u_{j,h}^{n+1} - u_{j,h}^n}  \notag \\
& \le \frac{\norm[\big]{u_{j,h}^{n+1} - u_{j,h}^n}^2}{2 \Delta t} + C \frac{\Delta t}{h} \norm[\big]{\nabla u_{j,h}^{n}}^2 \norm[\big]{\nabla u_{j,h}^{n'}}^2. \notag
\end{align}
\end{proof}
\subsection{Second-order schemes}
\label{sec:Stab_2nd}
In order to avoid non-essential technical difficulties, we will assume that $f_j = 0$. We note that all the proofs can be 
easily extended to $f_j \neq 0$ case with only small modifications.

The stability proofs and the timestep conditions depend on the ratios $\frac{\nu_\infty}{\nu_j}$ and on $\gamma$. 
The optimal value of $\gamma$ is determined through the following stability function:
\begin{equation}
 g_{\gamma} (x): = \frac{\left(x + \gamma - 1\right)^2}{4 \gamma \left( x-1 \right) + 2 \left( \gamma - 1 \right) + 2 x}. \label{eq:Stab_function}
\end{equation}
For given ensemble set, $\gamma \in [0,2)$ must be picked so that 
\begin{equation}
\sigma: = \max_{j} g_{\gamma} \left( \frac{\nu_\infty}{\nu_j} \right) \label{eq:sigma}
\end{equation}
belongs to $\left( \frac{1}{2}, 1 \right)$ and is as small as possible. This restricts the viscosities to satisfy
$\frac{\nu_\infty}{\nu_0} < 7$. 
\subsubsection{Stability with homogeneous Dirichlet boundary conditions}
\label{sec:StabDirichlet_2nd}
For the Algorithm \ref{thealgorithm4WK}, the time step condition takes the following form:
\begin{eqnarray}
\frac{\Delta t \left(1 + \gamma\right)^2 }{\nu_j \left(1+ 4 \gamma \right) \left(1-\sigma \right)}
\left( \norm[\big]{E_{j,h}^{n'}}_\infty + \frac{\mathrm{diam}\left( \Omega \right)}{d} \norm[\big]{\nabla \cdot E_{j,h}^{n '}}_\infty \right)^2 \le 1, \ \forall j = 1,...,J.
\label{eq:StabCondition_2nd}
\end{eqnarray}
\begin{thm}\label{StabilityLemma_2nd}
Let $\tilde{\nu}_j: = \nu_\infty + \left( \gamma-1 \right) \nu_j > 0$ and
\begin{align}
\mathrm{Ener}^{n+1}_j: & = \frac{\norm[\big]{u_{j,h}^{n+1} }^2 + \norm[\big]{E_{j,h}^{n+1}}^2}{4} + 
\gamma \frac{\norm[\big]{u_{j,h}^{n+1}-u_{j,h}^{n}}^2}{2} \notag \\
 & + \dt \frac{\tilde{\nu}_j}{2}\norm[\Big]{\nabla \left( u_{j,h}^{n+1} - u_{j,h}^{n} \right)}^2 + 
 \dt \sigma \nu_j \norm[\big]{\nabla u_{j,h}^{n+1}}^2. \label{eq:Ener_2nd}
\end{align}
If \eqref{eq:StabCondition_2nd} holds for each time step $n \ge 1$, then the solutions to Algorithm \ref{thealgorithm4WK} satisfy
\begin{align}
 \mathrm{Ener}^{N}_j \le \mathrm{Ener}^0_j \ \text{ for all } j.
\label{eq:Stability_2nd}
\end{align}
\end{thm}
\begin{proof}
Choose $v_h = u_{j,h}^{n+1} + \gamma \left(u_{j,h}^{n+1}-E_{j,h}^{n} \right)$ in \eqref{eq:Algth4aWK}, $q_h = p_{j,h}^{n+1}$ 
in \eqref{eq:Algth4bWK} and add them to get 
\begin{align}
\frac{\left( 3u_{j,h}^{n+1}-4u_{j,h}^n+u_{j,h}^{n-1},v_h\right)}{2\Delta t} & + b_1 \left(E_{j,h}^{n'}, E_{j,h}^n, v_h \right) \notag \\
& + \nu_\infty \left(\nabla \left(u_{j,h}^{n+1} - E^{n}_{j,h} \right), \nabla v_h \right) \label{eq:Stab1_2nd} \\
& + \nu_j \left(\nabla E_{j,h}^{n}, \nabla v_h \right) = 0.  \notag
\end{align}
Applying the standard identities repeatedly gives for the first term
\begin{align}
& \frac{\left( 3u_{j,h}^{n+1}-4u_{j,h}^n+u_{j,h}^{n-1},v_h\right)}{2\Delta t} 
= \frac{\left( 3u_{j,h}^{n+1}-4u_{j,h}^n+u_{j,h}^{n-1}, u_{j,h}^{n+1} \right)}{2\Delta t} \label{eq:Stab2_2nd} \\
& + \gamma \frac{\left( u_{j,h}^{n+1}-u_{j,h}^n,u_{j,h}^{n+1}-2u_{j,h}^n+u_{j,h}^{n-1}\right)}{\Delta t} + 
\gamma\frac{\norm[\big]{u_{j,h}^{n+1}-2u_{j,h}^n+u_{j,h}^{n-1}}}{2\Delta t} \notag \\
& = \frac{\norm[\big]{u_{j,h}^{n+1} }^2 + \norm[\big]{E_{j,h}^{n+1}}^2}{4 \Delta t} - \frac{\norm[\big]{u_{j,h}^{n} }^2 + \|E_{j,h}^{n} \|^2}{4 \Delta t}
 + \left( \gamma + \frac{1}{4} \right) \frac{\| u_{j,h}^{n+1}-E_{j,h}^n \|^2}{\Delta t} \notag \\
& + \gamma \frac{\| u_{j,h}^{n+1}-u_{j,h}^n\|^2 - \|u_{j,h}^n - u_{j,h}^{n-1}\|^2}{2 \Delta t}. \notag
\end{align}
Making use of the skew-symmetry of $b_1$ repeatedly, we can obtain a bound that is similar to 
the one we had in the proof of Theorem \ref{StabilityLemma}:
\begin{align}
b_1\left(E_{j,h}^{n'}, E_{j,h}^n, v_h \right) & = 
b_1\left(E_{j,h}^{n'}, E_{j,h}^{n}, u_{j,h}^{n+1} + \gamma \left(u_{j,h}^{n+1} - E^{n}_{j,h} \right) \right) \notag \\
& = \left( \gamma + 1 \right) b_1\left(E_{j,h}^{n'}, E_{j,h}^{n}, u_{j,h}^{n+1} \right) \notag \\
& = -\left( \gamma + 1 \right) b_1\left(E_{j,h}^{n'}, u_{j,h}^{n+1}, E_{j,h}^{n} \right) \notag \\
& = \left( \gamma + 1 \right) b_1\left(E_{j,h}^{n'}, u_{j,h}^{n+1}, u_{j,h}^{n+1} - E^{n}_{j,h} \right) \notag \\
& \le \left( \gamma + 1 \right) \norm[\big]{u_{j,h}^{n+1} - E_{j,h}^n} \norm[\big]{\nabla u_{j,h}^{n+1}}
\left( \norm[\big]{E_{j,h}^{n '}}_\infty + \frac{\norm[\big]{\nabla \cdot E_{j,h}^{n '}}_\infty}{2} \right) \notag \\
& \le \left( \gamma + \frac{1}{4} \right) \frac{\| u_{j,h}^{n+1}-E_{j,h}^n \|^2}{\Delta t} \label{eq:Stab3_2nd} \\
& + \frac{\Delta t \left( \gamma + 1 \right)^2}{1+4 \gamma}\left( \norm[\big]{E_{j,h}^{n '}}_\infty + \frac{\norm[\big]{\nabla \cdot E_{j,h}^{n '}}_\infty}{2} \right)^2
\norm[\big]{\nabla u_{j,h}^{n+1}}^2. \notag
\end{align}
Now we deal with the viscous terms. Again using the polarization identity, we get: 
\begin{align}
& \nu_\infty \left(\nabla \left(u_{j,h}^{n+1} - E^{n}_{j,h} \right), \nabla v_h \right) + \nu_j \left(\nabla E_{j,h}^{n}, \nabla v_h \right) \notag \\
& = \left(\nu_\infty-\nu_j \right) \left(\nabla \left(u_{j,h}^{n+1} - E^{n}_{j,h} \right), \nabla v_h \right) + \nu_j \left(\nabla u_{j,h}^{n+1}, \nabla v_h \right) \notag \\
& = \nu_j \norm[\big]{\nabla u_{j,h}^{n+1}}^2 + \gamma \left(\nu_\infty-\nu_j \right) \norm[\Big]{\nabla \left( u_{j,h}^{n+1}-E_{j,h}^{n} \right)}^2 \notag \\
& + \underbrace{\left(\nu_\infty + \left(\gamma-1 \right) \nu_j \right)}_{\tilde{\nu}_j} \left(\nabla \left(u_{j,h}^{n+1} - E^{n}_{j,h} \right), \nabla u_{j,h}^{n+1} \right)
\label{eq:Stab4_2nd} \\
& = \nu_j \norm[\big]{\nabla u_{j,h}^{n+1}}^2 + \gamma \left(\nu_\infty-\nu_j \right) \norm[\Big]{\nabla \left( u_{j,h}^{n+1}-E_{j,h}^{n} \right)}^2 \notag \\
& + \tilde{\nu}_j \left(\nabla \left(u_{j,h}^{n+1} - 2u^{n}_{j,h} + u^{n-1}_{j,h}\right), \nabla \left(u_{j,h}^{n+1}-u_{j,h}^{n} \right) \right) \notag \\
& + \tilde{\nu}_j \left(\nabla \left(u_{j,h}^{n+1} - 2u^{n}_{j,h} + u^{n-1}_{j,h}\right), \nabla u_{j,h}^{n} \right) \notag \\
& = \nu_j \norm[\big]{\nabla u_{j,h}^{n+1}}^2 + \gamma \left(\nu_\infty-\nu_j \right) \norm[\Big]{\nabla \left( u_{j,h}^{n+1}-E_{j,h}^{n} \right)}^2  \notag \\
& + \frac{\tilde{\nu}_j}{2} \left[ \norm[\Big]{\nabla \left( u_{j,h}^{n+1}-u_{j,h}^{n} \right)}^2 - 
\norm[\Big]{\nabla \left( u_{j,h}^{n}-u_{j,h}^{n-1} \right)}^2 + \norm[\Big]{\nabla \left( u_{j,h}^{n+1}-E_{j,h}^{n} \right)}^2 \right] \notag \\
& + \tilde{\nu}_j \left(\nabla \left(u_{j,h}^{n+1} - E^{n}_{j,h} \right), \nabla u_{j,h}^{n} \right). \notag 
\end{align}
The last term in \eqref{eq:Stab4_2nd} is a problematic term that prevents us from avoiding a restriction on the size of $\frac{\nu_\infty}{\nu_0}$.
It must be moved to the right-hand side and controlled by the remaining viscous terms. To this end, we first split the $\nu_j \norm[\big]{\nabla u_{j,h}^{n+1}}^2$ term in \eqref{eq:Stab4_2nd} into 
\begin{align}
\nu_j \norm[\big]{\nabla u_{j,h}^{n+1}}^2 & = (1-\sigma)\nu_j \norm[\big]{\nabla u_{j,h}^{n+1}}^2 + \sigma \nu_j \norm[\big]{\nabla u_{j,h}^{n+1}}^2 \label{eq:Stab5_2nd} \\
= (1-\sigma)\nu_j \norm[\big]{\nabla u_{j,h}^{n+1}}^2 & + \sigma \nu_j \left( \norm[\big]{\nabla u_{j,h}^{n+1}}^2 - \norm[\big]{\nabla u_{j,h}^{n}}^2 \right) 
+ \sigma \nu_j \norm[\big]{\nabla u_{j,h}^{n}}^2 \notag
\end{align}
and then apply the Cauchy-Schwarz and Youngs inequalities for the last term in \eqref{eq:Stab4_2nd}  
\begin{align*}
\tilde{\nu}_j \left| \left(\nabla \left(u_{j,h}^{n+1} - E^{n}_{j,h} \right), \nabla u_{j,h}^{n} \right) \right| & \le 
\left( \frac{\tilde{\nu}_j}{2} + \gamma \left( \nu_\infty - \nu_j \right) \right) \norm[\Big]{\nabla \left(u_{j,h}^{n+1} - E^{n}_{j,h} \right)}^2 \notag \\
& + \frac{\tilde{\nu}_j^2}{4\left(\frac{\tilde{\nu}_j}{2} + \gamma \left( \nu_\infty - \nu_j \right) \right)} \norm{\nabla u_{j,h}^{n}}^2.
\end{align*}
Putting all the ingredients together, we obtain the following inequality for the system:
\begin{align}
\mathrm{Ener}^{n+1}_j & - \mathrm{Ener}^{n}_j + 
\Delta t \left( \sigma \nu_j - \frac{\tilde{\nu}_j^2}{4 \gamma \left( \nu_\infty - \nu_j \right) + 2\nu_j \left( \gamma - 1 \right) + 2\nu_\infty}\right) \norm[\big]{\nabla u_{j,h}^{n}}^2
 \label{eq:Stab6_2nd} \\
& + \Delta t \left( (1-\sigma )\nu_j - \Delta t \gamma_1 \left( \norm[\big]{E_{j,h}^{n '}}_\infty 
+ \frac{\norm[\big]{\nabla \cdot E_{j,h}^{n '}}_\infty}{2} \right)^2 \right) \norm[\big]{\nabla u_{j,h}^{n+1}}^2 \le 0, \notag
\end{align}
where $\gamma_1: = \frac{ \left( \gamma + 1 \right)^2}{1+4 \gamma}.$
In order to have a stable scheme, both expressions in the parentheses must be non-negative. Consider the first expression:
\begin{align}
0 & \le \sigma \nu_j - \frac{\tilde{\nu}_j^2}{4 \gamma \left( \nu_\infty - \nu_j \right) + 2\nu_j \left( \gamma - 1 \right) + 2\nu_\infty}  
\Leftrightarrow \notag \\
0 & \le \sigma - \frac{\tilde{\nu}_j^2}{4 \gamma \left( \nu_\infty - \nu_j \right)\nu_j + 2\nu_j^2 \left( \gamma - 1 \right) + 2\nu_\infty \nu_j}  
\Leftrightarrow \label{eq:Stab7_2nd} \\
0 & \le \sigma - \frac{\left(\frac{\nu_\infty}{\nu_j} + \gamma - 1\right)^2}{4 \gamma \left( \frac{\nu_\infty}{\nu_j}-1 \right) + 2 \left( \gamma - 1 \right) + 2 \frac{\nu_\infty }{\nu_j}}
 = \sigma - g_\gamma \left( \frac{\nu_\infty}{\nu_j} \right). \notag
\end{align}
It can be verified that $g_\gamma \left( \frac{\nu_\infty}{\nu_j} \right) > \frac{1}{2}$, so that $\sigma > \frac{1}{2}$. On the other hand, 
to take larger timesteps we must have smaller $\sigma < 1$, which in turn means that we must pick $\gamma$ so that the stability function
\begin{figure}
\centering
\includegraphics[scale=0.75]{Ensemble_stability_function.eps}
\caption{The graphs of $g_\gamma (x)$ for $\gamma = 1, 1.5$ and $2$.}
\label{fig:Stab_function}
\end{figure}
has as small maximum value as possible for $x \in \left(1, \frac{\nu_\infty}{\nu_0} \right]$. Plots of $g_{\gamma} (x)$ for few values of 
$\gamma$ in Figure \ref{fig:Stab_function} reveal that we must have $\gamma < 2$, otherwise $g_{\gamma} (x) > 1$ as $x \rightarrow 1^{+}$. 
Moreover, since $g_{\gamma} (x) > 1$ as $x \rightarrow \infty$, the scheme is not guaranteed to be stable 
for very large values of $\frac{\nu_\infty}{\nu_0}$ (if greater than $7$).
With the definition of $\sigma$ in \eqref{eq:sigma}, the second parenthesis is also non-negative 
provided the timestep condition \eqref{eq:StabCondition_2nd} holds. Dropping the last two terms in 
\eqref{eq:Stab6_2nd}, and proceeding inductively we arrive at the desired inequality. 
\end{proof}
\subsubsection{Stability with outflow boundary conditions}
\label{sec:StabilityOutflow_2nd}
Now we consider the case of $\Gamma_N \neq \emptyset$.
The stability for the Algorithm \ref{thealgorithm5WK} holds under the following two timestep conditions
\begin{eqnarray}
\frac{\Delta t \left(1 + \gamma\right)^2 }{\nu_j \left(1+ 4 \gamma \right) \left(1-\sigma \right)}
\left( \norm[\big]{E_{j,h}^{n'}}_\infty + \frac{1}{\sqrt{\lambda_1}} \norm[\big]{\nabla \cdot E_{j,h}^{n '}}_\infty \right)^2  &\le 1, \ \forall j = 1,...,J,
\label{eq:StabConditionOBC1_2nd} \\
\frac{\dt \left( \gamma + 1 \right)^2}{\left( 4\gamma + 1 \right) \nu_j} \norm[\Big]{ \left(E_{j,h}^{n'} \cdot \bfn \right) \MTheta_0\left(E_{j,h}^{n'} \cdot \bfn \right)}^2_{\infty,\Gamma_N} & \le 1,  \ \forall j = 1,...,J.
\label{eq:StabConditionOBC2_2nd}
\end{eqnarray}
\begin{thm}\label{StabilityLemmaOBC_2nd}
\begin{align}
\mathrm{Ener}^{n+1}_j: & = \frac{\norm[\big]{u_{j,h}^{n+1} }^2 + \norm[\big]{E_{j,h}^{n+1}}^2}{4} + 
\gamma \frac{\norm[\big]{u_{j,h}^{n+1}-u_{j,h}^{n}}^2}{2} \notag \\
 & + \LL \frac{\norm[\big]{u_{j,h}^{n+1} }^2_{\Gamma_N} + \norm[\big]{E_{j,h}^{n+1}}^2_{\Gamma_N}}{4} + 
\LL \gamma \frac{\norm[\big]{u_{j,h}^{n+1}-u_{j,h}^{n}}^2_{\Gamma_N}}{2} \notag \\
 & + \dt \frac{\tilde{\nu}_j}{2}\norm[\Big]{\nabla \left( u_{j,h}^{n+1} - u_{j,h}^{n} \right)}^2 + \dt \sigma \nu_j \norm[\big]{\nabla u_{j,h}^{n+1}}^2. \label{eq:Ener_2ndOBC}
\end{align}
If for each time step $n \ge 1$, the conditions \eqref{eq:StabConditionOBC1_2nd}-\eqref{eq:StabConditionOBC2_2nd} hold,
then the solutions to Algorithm \ref{thealgorithm5WK} satisfy
\begin{align}
\mathrm{Ener}^{N}_j & + \Delta t \sum\limits_{n=0}^{N-1} \bigintsss\limits_{\Gamma_N} \left[ \frac{\overline{E^n_h} \cdot \bfn}{2} \left|E_{j,h}^{n+1} \right|^2 
\MTheta_1\left(\overline{E^n_h} \cdot \bfn\right)  + 
\frac{u^{n'}_{j,h} \cdot \bfn}{2} \left|E_{j,h}^{n}\right|^2 \MTheta_1\left(E^{n'}_{j,h} \cdot \bfn \right)  \right] \notag \\
& \le \exp\left( \frac{5\nu_j T}{\mathrm{L}} \right) \mathrm{Ener}^0_j.
\label{eq:StabilityOBC_2nd}
\end{align}
\end{thm}
\begin{proof}
Choose the test functions $v_h = u_{j,h}^{n+1} + \gamma \left(u_{j,h}^{n+1}-E_{j,h}^{n} \right)$, $q_h = p^{n+1}_{j,h}$, and add the equations \eqref{eq:Algth5aWK}-\eqref{eq:Algth5bWK}. 
By applying \eqref{eq:NL_ener}, the first nonlinear term becomes
\begin{align}
& b_1 \left(\overline{E_h^n}, u_{j,h}^{n+1} + \gamma \left(u_{j,h}^{n+1}-E_{j,h}^{n} \right), v_h \right)
+ b_1 \left(E_{j,h}^{n'}, E_{j,h}^{n} , v_h \right) \notag \\
& = b_1 \left(\overline{E_h^n}, v_h, v_h \right) + b_1 \left( E_{j,h}^{n'}, E_{j,h}^{n}, E_{j,h}^{n}\right) 
+ \left( \gamma+1 \right) b_1 \left(E_{j,h}^{n'}, E_{j,h}^{n}, u_{j,h}^{n+1}-E_{j,h}^{n}  \right) \notag \\
& = \bigintsss\limits_{\Gamma_N} \left[ \frac{ \overline{E^n_h} \cdot \bfn}{2} \left|u_{j,h}^{n+1} + \gamma \left(u_{j,h}^{n+1}-E_{j,h}^{n} \right) \right|^2 + 
\frac{E^{n'}_{j,h} \cdot \bfn}{2} \left|E_{j,h}^{n} \right|^2 \right] \notag \\
& + \left( \gamma+1 \right) b_1 \left(E_{j,h}^{n'}, E_{j,h}^{n}, u_{j,h}^{n+1}-E_{j,h}^{n}  \right), \label{eq:Stab1OBC_2nd}
\end{align}
and similarly, 
\begin{align}
& b_2 \left(\overline{E_h^n}, u_{j,h}^{n+1} + \gamma \left(u_{j,h}^{n+1}-E_{j,h}^{n} \right), v_h \right)
+ b_2 \left(E_{j,h}^{n'}, E_{j,h}^{n} , v_h \right) \notag \\
& = - \bigintsss\limits_{\Gamma_N} \frac{ \overline{E^n_h} \cdot \bfn}{2} \left|u_{j,h}^{n+1} + \gamma \left(u_{j,h}^{n+1}-E_{j,h}^{n} \right) \right|^2 
\MTheta_0(\overline{E^n_h} \cdot \bfn)  \label{eq:Stab2OBC_2nd} \\
& - \bigintsss\limits_{\Gamma_N} \frac{E^{n'}_{j,h} \cdot \bfn}{2} \left|E_{j,h}^{n} \right|^2 \MTheta_0(E^{n'}_{j,h} \cdot \bfn)  
+ \left( \gamma+1 \right) b_2 \left( E_{j,h}^{n'}, E_{j,h}^{n},u_{j,h}^{n+1} - E_{j,h}^{n}\right). \notag
\end{align}
Bounding the $b_1\left( \cdot, \cdot, \cdot \right)$ and viscous terms as in the proof of the Theorem $\ref{StabilityLemma_2nd}$, 
and taking \eqref{eq:Stab1OBC_2nd}-\eqref{eq:Stab2OBC_2nd} into account gives
\begin{align}
& \mathrm{Ener}^{n+1}_j - \mathrm{Ener}^{n}_j + \LL \left( \gamma + \frac{1}{4} \right) \frac{\| u_{j,h}^{n+1}-E_{j,h}^n \|_{\Gamma_N}^2}{\Delta t} 
+ \notag \\
& \underbrace{ \bigintsss\limits_{\Gamma_N} \left[ \frac{ \overline{E^n_h} \cdot \bfn }{2} \left|u_{j,h}^{n+1} + \gamma \left(u_{j,h}^{n+1}-E_{j,h}^{n} \right) \right|^2 \MTheta_1\left(\overline{u^n_h} \cdot \bfn \right) + 
\frac{E^{n'}_{j,h} \cdot \bfn}{2} \left|E_{j,h}^{n} \right|^2 \MTheta_1\left(E^{n'}_{j,h} \cdot \bfn \right)  \right]}_{:=F_{n+1} \ge 0} \label{eq:Stab3OBC_2nd} \\
& = - \left( \gamma+1 \right) b_2 \left( E_{j,h}^{n'}, E_{j,h}^{n},u_{j,h}^{n+1} - E_{j,h}^{n} \right). \notag
\end{align}
For the remaining term on the right-hand side of \eqref{eq:Stab3OBC_2nd}, we apply the Cauchy-Schwarz to get
\begin{align}
& \left( \gamma+1 \right) b_2 \left( E_{j,h}^{n'}, E_{j,h}^{n},u_{j,h}^{n+1} - E_{j,h}^{n} \right) \notag \\
& = \left( \gamma+1 \right)  \bigintsss\limits_{\Gamma_N} \MTheta_0(E^{n'}_{j,h} \cdot \bfn) \frac{- E^{n'}_{j,h} \cdot \bfn }{2} E_{j,h}^{n} \cdot \left(u_{j,h}^{n+1} - E_{j,h}^{n} \right) 
 \notag \\
& \le \mathrm{L}\left(\gamma + \frac{1}{4} \right) \frac{\norm[\big]{ u_{j,h}^{n+1} - E_{j,h}^{n} }^2_{\Gamma_N}}{\Delta t} \notag \\
& + \frac{\dt \left( \gamma + 1 \right)^2 \norm[\Big]{\left(E_{j,h}^{n'} \cdot \bfn \right) \MTheta_0\left(E_{j,h}^{n'} \cdot \bfn \right)}^2_{\infty,\Gamma_N}}{\left( 4\gamma + 1 \right) \mathrm{L}} 
\norm[\big]{ E_{j,h}^{n} }^2_{\Gamma_N} \notag \\
& \le \mathrm{L}\left(\gamma + \frac{1}{4} \right) \frac{\norm[\big]{ u_{j,h}^{n+1} - E_{j,h}^{n} }^2_{\Gamma_N}}{\Delta t} + 
\frac{\nu_j}{\mathrm{L}} \norm[\big]{ E_{j,h}^{n} }^2_{\Gamma_N}. \label{eq:Stab4OBC_2nd} 
\end{align}
The last bounded has been obtained under \eqref{eq:StabConditionOBC2_2nd}. Putting everything together and summing over the timesteps yields
\begin{align}
\mathrm{Ener}^N_j + \dt \sum \limits_{n=1}^{N-1} F_{n+1} \le 
\mathrm{Ener}^0_j + \frac{5 \nu_j}{\mathrm{L}} \dt \sum \limits_{n=1}^{N-1}\norm[\big]{u_{j,h}^{n}}^2_{\Gamma_N}. 
\label{eq:Stab5OBC_2nd}
\end{align}
Gronwall's inequality completes the proof.
\end{proof}
\color{black}
Now we prove the stability of the Algorithm \ref{thealgorithm6WK} under 
\begin{eqnarray}
C\frac{\left(\gamma+1 \right)^2}{1+4 \gamma} \frac{\Delta t}{h \left( 1-\sigma \right) \nu_j} \norm[\big]{\nabla E_{j,h}^{n '}}^2 \le 1, \ \forall j = 1,...,J, \ C = \mathcal{O}(1).
\label{eq:StabConditionOBC3_2nd}
\end{eqnarray}
\begin{thm}\label{StabilityLemmaOBC3_2nd}
Let 
\begin{align}
\mathrm{Ener}^{n+1}_j: & = \frac{\norm[\big]{u_{j,h}^{n+1} }^2 + \norm[\big]{E_{j,h}^{n+1}}^2}{4} + 
\gamma \frac{\norm[\big]{u_{j,h}^{n+1}-u_{j,h}^{n}}^2}{2} \notag \\
 & + \dt \frac{\tilde{\nu}_j}{2}\norm[\Big]{\nabla \left( u_{j,h}^{n+1} - u_{j,h}^{n} \right)}^2 + 
 \dt \sigma \nu_j \norm[\big]{\nabla u_{j,h}^{n+1}}^2. \label{eq:Ener_2nd_OBC3}
\end{align}
If for each time step $n \ge 1$, the condition \eqref{eq:StabConditionOBC3_2nd} holds, then the solutions to Algorithm \ref{thealgorithm6WK} satisfy
\begin{align}
 \mathrm{Ener}^{N}_j + 
 \Delta t \sum\limits_{n=0}^{N-1} \bigintsss\limits_{\Gamma_N} \frac{\overline{E^n_h} \cdot \bfn}{2} 
 \left|u_{j,h}^{n+1} + \gamma \left( u_{j,h}^{n+1} - E_{j,h}^{n} \right)\right|^2 \MTheta_1 \left(\overline{E^n_h} \cdot \bfn \right) 
 \le \mathrm{Ener}^0_j.
\label{eq:StabilityOBC3}
\end{align}
\end{thm}
\begin{proof}
To prove \eqref{eq:StabilityOBC3}, we combine the ideas used in the proofs of Theorems \ref{StabilityLemmaOBC3}, 
\ref{StabilityLemma_2nd}, and \ref{StabilityLemmaOBC_2nd}. Taking advantage of the skew-symmetry of $b_3 \left( \cdot, \cdot, \cdot \right)$, 
we have that:
\begin{align}
& b_3\left(E_{j,h}^{n'}, E_{j,h}^n, u_{j,h}^{n+1} + \gamma \left( u_{j,h}^{n+1} - E_{j,h}^{n} \right) \right) \notag \\
& = \left( \gamma + 1 \right) b_3\left(E_{j,h}^{n'}, E_{j,h}^n, u_{j,h}^{n+1} \right) \label{eq:Stab2OBC2} \\
& = \left( \gamma + 1 \right) b_3\left(E_{j,h}^{n'}, u_{j,h}^{n+1}, u_{j,h}^{n+1} - E_{j,h}^{n} \right) \notag \\
& \le \left( \gamma + \frac{1}{4} \right) \frac{\norm[\big]{u_{j,h}^{n+1} - E_{j,h}^n}^2}{\Delta t} 
+ C \frac{\Delta t}{h} \frac{\left(\gamma+1 \right)^2}{1+4 \gamma} \norm[\big]{\nabla u_{j,h}^{n}}^2 \norm[\big]{\nabla E_{j,h}^{n'}}^2. \notag
\end{align}
\end{proof}
\subsection{Convergence with homogeneous Dirichlet boundary conditions}
\label{sec:Conv}
Convergence for the case of $\nu_j = \nu$ has been already been considered in the literature \cite{Jiang_2014}, 
and the effect of additional viscous terms are straightforward to analyze.
\begin{thm}\label{convthm} Let $(X_h,Q_h) = (P_{2},P_1)$, be a Taylor-Hood pair.
Assuming enough smoothness on the exact solution and the timestep condition \eqref{eq:StabConditionDissipative},
the velocity error $e^n_j := u_j(x,t^n)-u^n_{j,h}$ in the Algorithm \ref{thealgorithm4WK} satisfies the following error estimate:
\begin{align}
 \| e^{N}_j \|^{2} + \sigma \nu_j \Delta t \sum \limits_{n=1}^{N} \| \nabla e^n_j \|^2 \le C(T,\nu_\infty-\nu_j) \left(h^{4} + \Delta t^4 \right).
\label{eq:Err}
\end{align}
\end{thm}
\section{Numerical Experiments}
\label{sec:NumEx}
The simulations are performed using the FreeFem++ \cite{MR3043640} package, with the $(P_2,P_1)$ used to approximate the velocity and pressure spaces, 
respectively. We only tested second order schemes. All the linear systems are solved using direct solvers.

In the last two channel flow examples, for Algorithm \ref{thealgorithm5WK}, we tested few different values of $\mathrm{L}$. Namely, we set $\mathrm{L} = \tau \mathrm{D}$, where 
$\tau \le 1$, and $\mathrm{D}$ is the inlet diameter of a channel. Larger values of $\tau$ altered the solution 
qualitatively near the outlet, as was also observed in \cite{DONG2015300}, and therefore we used $\tau = 0.01$ in both cases.
\subsection{Convergence study}
We first confirm the predicted convergence rates with different number of ensemble cases, 
i.e., different values of $J$. Namely, we tested three different cases: $J=2, J=4$ and $J=8$, with appropriate values of $\gamma$. 
Additionally, for the $J=8$ case, we compared the errors with different values of $\gamma$; and with the errors obtained from independent simulations. 
The domain is taken as $\Omega=(0,1)^2$, viscosity is $\nu = 1$ and the final time $T=1$. 
Initial size of the mesh is $h=\frac{1}{10}$. If the stability condition \eqref{eq:StabCondition_2nd} is violated, then the timestep is halved. 
The final time step and the initial time step are the same for all the meshes: 
$h=\frac{1}{10}$ with $\Delta t=0.01$, $h=\frac{1}{20}$ with $\Delta t=0.025$, $h=\frac{1}{40}$ with $\Delta t=0.00625$, and 
$h=\frac{1}{80}$ with $\Delta t=0.0015625$.

We generate perturbations using
 \begin{align*}
u & = \left( \begin{array}{c}
x^2-y\sin(t)  \\
-2xy+x\cos(t)
 \end{array}\right), p = (x+y-1)\sin(t).  \label{eq:GT}
\end{align*}
Namely, the ensemble solutions, viscosities and the source terms are taken to be 
\begin{align*}
u_{j} & = (1 + j \varepsilon) u,\ p_{1} = (1+ j\varepsilon), \nu_j = (1 + j \varepsilon) \nu, j=\overline{1,\frac{J}{2}} \\
u_{j} & = (1 - j\varepsilon) u,\ p_{j} = (1-j\varepsilon), \nu_j = (1 - j \varepsilon) \nu, j=\overline{\frac{J}{2}+1,J} \\
f_{j} & = \partial_t u_j + u_j \cdot \nabla u_j - \nu_j \Delta u_j + \nabla p_j, \ j=\overline{1,J},
\end{align*}
where we alter the value of $\varepsilon$ for each case, as specified below.

For $J=2$ case, we pick the perturbation parameter $\varepsilon = 0.1$.
Since $\frac{\nu_\infty}{\nu_0}=1.222 < 4$ and $\sigma=0.75$, in this case we can pick $\gamma=1.5$, cf. Fig. \ref{fig:Stab_function}. 
The errors are reported in the Table \ref{tab:ErrUP2}, which give second order convergence, as expected. 

For $J=4$ case, we pick the perturbation parameter $\varepsilon = 0.2$. 
Since $\frac{\nu_\infty}{\nu_0}=2.33333<4$ and $\sigma=0.75$,
we can again pick $\gamma=1.5$. The errors are reported in the Tables \ref{tab:ErrU4} and \ref{tab:ErrP4}. 

For $J=8$ case, we take $\varepsilon = 0.2$, so that we can test a larger set of viscosities.
Note that in this case $\frac{\nu_\infty}{\nu_0} = 9 > 7$. We also compare results with different $\gamma$'s. 
In the Tables \ref{tab:ErrU811} -\ref{tab:ErrP812}, we report the results with $\gamma=1$, where 
we also performed sequential runs for comparison (denoted by tilde notation). The results are nearly identical. 
Results with $\gamma=1.5$ and $\gamma=2$ are shown in the Tables
\ref{tab:ErrU82}-\ref{tab:ErrP83}, which again yield similar accuracy. Moreover, the Tables indicate
the expected second order convergence rate.

\begin{table}
\centering
\caption{Velocity and pressure errors in $L^2$ norm for Algorithm \ref{thealgorithm4WK} with $J=2$}
\tabcolsep=0.01cm
\begin{tabular}{|c|c|c|c|c|c|c|c|c|c|c|}
\hline
 $ h $ & $\|u_{1} - u_{1,h}\|$ & \text{rate} & $\|p_{1} - p_{1,h}\|$  & \text{rate} & $\|u_{2} - u_{2,h}\|$ & \text{rate} & $\|p_{2} - p_{2,h}\|$ & \text{rate} \\
\hline
$\frac{1}{10}$ & $0.0634833$ & $1.9996$ & $0.0233303 $ & $1.97047$ & $0.0519409 $ & $1.9996$ &  $0.0193605$ & $1.98161$   \\
\hline
$\frac{1}{20}$ & $0.0158767 $ & $1.99997$ &  $0.00600868$ & $1.99286 $ & $0.0129901$ & $1.99997$ & $0.00493038$ & $1.99507$  \\
\hline
$\frac{1}{40}$ & $0.00396928$ & $1.99999$ & $0.00151295$ & $1.99817$ & $0.00324759$ & $1.99999$ & $0.0012387$& $1.99874$  \\
\hline
$\frac{1}{80}$ & $0.000992321$ & $-$ & $0.000378929 $ & $-$ & $0.000811899$ & $-$ & $0.000310065$& $-$  \\
\hline
\end{tabular}
\label{tab:ErrUP2}
\end{table}
\begin{table}
\centering
\caption{Velocity errors in $L^2$ norm for Algorithm \ref{thealgorithm4WK} with $J=4$}
\tabcolsep=0.01cm
\begin{tabular}{|c|c|c|c|c|c|c|c|c|c|c|}
\hline
 $ h $ & $\|u_{1} - u_{1,h}\|$ & \text{rate} & $\|u_{2} - u_{2,h}\||$  & \text{rate} & $\|u_{3} - u_{3,h}\|$ & \text{rate} & $\|u_{4} - u_{4,h}\|$ & \text{rate} \\
\hline
$\frac{1}{10}$ & $0.0807969$ & $1.99962$ & $0.0692545 $ & $1.99963$ & $0.0461717 $ & $1.99967$ &  $0.0346292$ & $1.99967$   \\
\hline
$\frac{1}{20}$ & $0.0202068 $ & $1.99997$ &  $0.0173201$ & $1.99997 $ & $0.0115467$ & $1.99997$ & $0.00866002$ & $1.99997$  \\
\hline
$\frac{1}{40}$ & $0.00505181$ & $2.00000$ & $0.00433012$ & $2.0000$ & $0.00288675$ & $1.99999$ & $0.00216506$& $1.99999$  \\
\hline
$\frac{1}{80}$ & $0.00126295$ & $-$ & $0.00108253$ & $-$ & $0.000721688$ & $-$ & $0.000541266 $& $-$  \\
\hline
\end{tabular}
\label{tab:ErrU4}
\end{table}
\begin{table}
\centering
\caption{Pressure errors in $L^2$ norm for Algorithm \ref{thealgorithm4WK} with $J=4$}
\tabcolsep=0.01cm
\begin{tabular}{|c|c|c|c|c|c|c|c|c|}
\hline
 $h$& $\|p_{1} - p_{1,h}\|$ & \text{rate} & $\|p_{2} - p_{2,h}\|$ & \text{rate} & $\|p_{3} - p_{3,h}\|$ & \text{rate} & $\|p_{4} - p_{4,h}\|$ & \text{rate} \\
\hline
$\frac{1}{10}$ & $0.0290629$ &$1.95368$ &  $0.025271$ & $1.96490$ & $0.017328$ & $1.98699$ & $0.0131769$ & $1.99790$  \\
\hline
$\frac{1}{20}$ & $0.00761436 $ & $1.98956$  & $0.00654547$ & $1.99176$ & $0.00438888 $ & $1.99616$ & $0.00330115$ & $1.99836$  \\
\hline
$\frac{1}{40}$ & $0.00192362 $ & $1.99738$ & $0.00164994$ & $1.99800$ & $0.00110145$ & $1.99901$ & $0.000826642$  & $1.99951$    \\
\hline
$\frac{1}{80}$ & $0.000482168$ & $-$ & $0.000413313$ & $-$ & $0.000275635$ & $-$ & $0.000206762$  & $-$    \\
\hline
\end{tabular}
\label{tab:ErrP4}
\end{table}
\begin{table}
\caption{Velocity errors in $L^2$ norm for Algorithm \ref{thealgorithm4WK} with $J=8$ and $\gamma=1$}
\resizebox{\textwidth}{!}{
\centering
\tiny\addtolength{\tabcolsep}{-5pt}
\begin{tabular}{|c|c|c|c|c|c|c|c|c|c|c|c|c|c|c|c|c|c|c|c|c|c|}
\hline
 $h$ & $\|u_{1} - u_{1,h}\|$ & \text{rate} & $\|u_{1} - \tilde{u}_{1,h}\|$ & \text{rate} & $\|u_{2} - u_{2,h}\|$ & \text{rate} & $\|u_{2} - \tilde{u}_{2,h}\|$ & \text{rate} & $\|u_{3} - u_{3,h}\|$ & \text{rate} & $\|u_{3} - \tilde{u}_{3,h}\|$ & \text{rate} & $\|u_{4} - u_{4,h}\|$ & \text{rate} & $\|u_{4} - \tilde{u}_{4,h}\|$ & \text{rate} \\
\hline
$\frac{1}{10}$ & $0.0692548$ & $1.99963$ & $0.0692548 $ & $1.99963$ & $0.0807968 $ & $1.99962$ &  $0.0807968$ & $1.99962$& $0.0923393$ & $1.99962$ & $0.0923393 $ & $1.99963$ & $0.103882 $ & $1.99963$ &  $0.103882$ & $1.99963$   \\
\hline
$\frac{1}{20}$ & $0.0173201$ & $1.99998$ &  $0.0173201$ & $1.99998 $ & $0.0202068$ & $1.99998$ & $0.0202068$ & $1.99998$ 
& $0.0230935$ & $1.99998$ &  $0.0230934$ & $1.99998 $ & $0.0259802$ & $1.99998$ & $0.0259801$ & $1.99998$  \\
\hline
$\frac{1}{40}$ & $0.00433012$ & $2.0000$ &  $0.00433012$ & $2.0000 $ & $0.00505181$ & $2.0000$ & $0.00505181$ & $2.0000$ 
& $0.0057735$ & $1.99999$ &  $0.00577349$ & $1.99999 $ & $0.00649518$ & $1.99999$ & $0.00649518$ & $1.99999$  \\
\hline
$\frac{1}{80}$ & $0.00108253$ & $-$ &  $0.00108253$ & $-$ & $0.00126295$ & $-$ & $0.00126295$ & $-$ 
& $0.00144338$ & $-$ &  $0.00144338$ & $- $ & $0.0016238$ & $-$ & $0.0016238$ & $-$  \\
\hline
\end{tabular}
\label{tab:ErrU811}}
\end{table}
\begin{table}
\caption{Velocity errors in $L^2$ norm for Algorithm \ref{thealgorithm4WK} with $J=8$ and $\gamma=1$, cont-d.}
\resizebox{\textwidth}{!}{
\centering
\tabcolsep=0.01cm
\begin{tabular}{|c|c|c|c|c|c|c|c|c|c|c|c|c|c|c|c|c|c|c|c|c|c|c|c|}
\hline
 $h$ & $\|u_{5} - u_{5,h}\|$ & \text{rate} & $\|u_{5} - \tilde{u}_{5,h}\|$ & \text{rate} & $\|u_{6} - u_{6,h}\|$ & \text{rate} & $\|u_{6} - \tilde{u}_{6,h}\|$ & \text{rate} & $\|u_{7} - u_{7,h}\|$ & \text{rate} & $\|u_{7} - \tilde{u}_{7,h}\|$ & \text{rate} & $\|u_{8} - u_{8,h}\|$ & \text{rate} & $\|u_{8} - \tilde{u}_{8,h}\|$ & \text{rate} \\
\hline
$\frac{1}{10}$ & $0.0461732$ & $1.99971$ & $0.0461734  $ & $1.99971$ & $0.0346236 $ & $1.99952$ &  $0.0346241$ & $1.99953$& $0.0230708$ & $1.99902$ & $0.0230704 $ & $1.99900$ & $0.0115419 $ & $1.99956$ &  $0.0115393$ & $1.99933$   \\
\hline
$\frac{1}{20}$ & $0.0115467 $ & $1.99997$ &  $0.0115467$ & $1.99997 $ & $0.00866003$ & $1.99998$ & $0.00866005$ & $1.99998$ 
& $0.00577335$ & $1.99998$ &  $0.00577337$ & $1.99998$ & $0.00288673$ & $1.99999$ & $0.00288676$ & $2.0000$  \\
\hline
$\frac{1}{40}$ & $0.00288675$ & $1.99999$ &  $0.00288675 $ & $1.99999$ & $0.00216506$ & $1.99999$ & $0.00216506$ & $1.99999$ 
& $0.00144337$ & $1.99999$ &  $0.00144337$ & $1.99999$ & $0.000721686$ & $1.99999$ & $0.000721688$ & $2.0000$  \\
\hline
$\frac{1}{80}$ & $0.000721688 $ & $-$ &  $0.000721688$ & $-$ & $0.000541266$ & $-$ & $0.000541266$ & $-$ 
& $0.000360844$ & $-$ &  $0.000360844$ & $-$ & $0.000180422$ & $-$ & $0.000180422$ & $-$  \\
\hline
\end{tabular}
\label{tab:ErrU812}}
\end{table}
\begin{table}
\caption{Pressure errors in $L^2$ norm for Algorithm \ref{thealgorithm4WK} with $J=8$ and $\gamma=1$}
\resizebox{\textwidth}{!}{
\centering
\tabcolsep=0.01cm
\begin{tabular}{|c|c|c|c|c|c|c|c|c|c|c|c|c|c|c|c|c|c|c|c|c|c|c|c|}
\hline
$h$& $\|p_{1} - p_{1,h}\|$ & \text{rate} & $\|p_{1} - \tilde{p}_{1,h}\|$ & \text{rate} & $\|p_{2} - p_{2,h}\|$ & \text{rate} & $\|p_{2} - \tilde{p}_{2,h}\|$ & \text{rate} & $\|p_{3} - p_{3,h}\|$ & \text{rate} & $\|p_{3} - \tilde{p}_{3,h}\|$ & \text{rate} & $\|p_{4} - p_{4,h}\|$ & \text{rate} & $\|p_{4} - \tilde{p}_{4,h}\|$ & \text{rate} \\
\hline
$\frac{1}{10}$ & $0.0253294$ & $1.96573$ & $0.0252823 $ & $1.96054$ & $0.0291269 $ & $1.95439$ &  $0.0290299$ & $1.95342$& $0.0328058$ & $1.94300$ & $0.032661 $ & $1.94208$ & $0.0363675 $ & $1.93156$ &  $0.0361833$ & $1.93110$   \\
\hline
$\frac{1}{20}$ & $0.00655508$ & $1.99280$ &  $0.00654738$ & $1.99196 $ & $0.00762551$ & $1.99059$ & $0.00760774$ & $1.98895$ 
& $0.00868967$ & $1.98839$ &  $0.00865954$ & $1.98595$ & $0.00974758$ & $1.98617$ & $0.00970286$ & $1.98295$  \\
\hline
$\frac{1}{40}$ & $0.00165064$ & $1.99828$ &  $0.00165009$ & $1.99806$ & $0.00192444$ & $1.99769$ & $0.00192313$ & $1.99721$ 
& $0.00219787$ & $1.99727$ &  $0.00219563$ & $1.99648$ & $0.00247094$ & $1.99675$ & $0.0024676$ & $1.99576$  \\
\hline
$\frac{1}{80}$ & $0.000413372$ & $-$ &  $0.000413326 $ & $- $ & $0.000482222$ & $-$ & $0.000482128$ & $-$ 
& $0.00055097$ & $-$ &  $0.000550843$ & $-$ & $0.000619747$ & $-$ & $0.000619523$ & $-$  \\
\hline
\end{tabular}
\label{tab:ErrP811}}
\end{table}
\begin{table}
\caption{Pressure errors in $L^2$ norm for Algorithm \ref{thealgorithm4WK} with $J=8$ and $\gamma=1$, cont-d.}
\resizebox{\textwidth}{!}{
\centering
\tabcolsep=0.01cm
\begin{tabular}{|c|c|c|c|c|c|c|c|c|c|c|c|c|c|c|c|c|c|c|c|c|c|c|c|}
\hline
 $h$& $\|p_{5} - p_{5,h}\|$ & \text{rate} & $\|p_{5} - \tilde{p}_{5,h}\|$ & \text{rate} & $\|p_{6} - p_{6,h}\|$ & \text{rate} & $\|p_{6} - \tilde{p}_{6,h}\|$ & \text{rate} & $\|p_{7} - p_{7,h}\|$ & \text{rate} & $\|p_{7} - \tilde{p}_{7,h}\|$ & \text{rate} & $\|p_{8} - p_{8,h}\|$ & \text{rate} & $\|p_{8} - \tilde{p}_{8,h}\|$ & \text{rate} \\
\hline
$\frac{1}{10}$ & $0.0173692$ & $1.98789$ & $0.0174076   $ & $1.98891$ & $0.0132118$ & $1.99907$ &  $0.0132718$ & $2.00121$& $0.00893577$ & $2.01064$ & $0.00899729 $ & $2.01391$ & $0.00453081 $ & $2.02183$ &  $0.0045758$ & $2.02697$   \\
\hline
$\frac{1}{20}$ & $0.00439534 $ & $1.99721$ &  $0.00440058$ & $1.99805$ & $0.00330602$ & $1.99941$ & $0.00331394$ & $2.00112$ 
& $0.00221037$ & $2.00162$ &  $0.00221835$ & $2.00419 $ & $0.00110837$ & $2.00382$ & $0.00111371$ & $2.00727$  \\
\hline
$\frac{1}{40}$ & $0.00110191$ & $1.99932$ &  $0.00110229 $ & $1.99954 $ & $0.000826993$ & $1.99980$ & $0.000827559$ & $2.00030$ 
& $0.0005517$ & $2.00030$ &  $0.000552267$ & $2.00110 $ & $0.000276038$ & $2.00083$ & $0.000276415$ & $2.00186$  \\
\hline
$\frac{1}{80}$ & $0.000275665 $ & $-$ &  $0.000275699 $ & $-$ & $0.000206789$ & $-$ & $0.000206826$ & $-$ 
& $0.000137883$ & $-$ &  $0.000137915$ & $-$ & $6.89516e-05$ & $-$ & $6.89751e-05$ & $-$  \\
\hline
\end{tabular}
\label{tab:ErrP812}}
\end{table}
\begin{table}
\caption{Velocity errors in $L^2$ norm for Algorithm \ref{thealgorithm4WK} with $J=8$ and $\gamma=1.5$}
\resizebox{\textwidth}{!}{
\centering
\tabcolsep=0.01cm
\begin{tabular}{|c|c|c|c|c|c|c|c|c|c|c|c|c|c|c|c|c|c|c|c|c|c|c|c|}
\hline
 $\Delta t$ & $\|u_{1} - u_{1,h}\|$ & \text{rate} & $\|u_{2} - u_{2,h}\|$  & \text{rate} & $\|u_{3} - u_{3,h}\|$ & \text{rate} & $\|u_{4} - u_{4,h}\|$  & \text{rate} & $\|u_{5} - u_{5,h}\|$ & \text{rate} & $\|u_{6} - u_{6,h}\|$  & \text{rate} & $\|u_{7} - u_{7,h}\|$ & \text{rate} & $\|u_{8} - u_{8,h}\|$  & \text{rate}\\
\hline
$\frac{1}{10}$ & $0.0692549$ & $1.99963$ & $0.0807968 $ & $1.9996$ & $0.0923393 $ & $1.99962$ &  $0.103882$ & $1.99963$
& $0.046173$ & $1.99970$ & $0.0346233$ & $1.99952$ & $0.0230709  $ & $1.99902$ &  $0.0115428$ & $1.99964$   \\
\hline
$\frac{1}{20}$ & $0.0173201$ & $1.99998$ & $0.0202068 $ & $1.99998$ & $0.0230935 $ & $1.99998$ &  $0.0259802$ & $1.99998$
& $0.0115467$ & $1.99998$ & $0.00866002$ & $1.99997$ & $0.00577334  $ & $1.99998$ &  $0.00288673$ & $1.99999$   \\
\hline
$\frac{1}{40}$ & $0.00433012$ & $2.0000$ & $0.00505181$ & $2.0000$ & $0.0057735$ & $1.99999$ & $0.00649518$& $1.99999$ 
& $0.00288675$ & $1.99999$ & $0.00216506$ & $1.99999$ & $0.00144337$ & $1.99999$ & $0.000721685$& $1.99999$ \\
\hline
$\frac{1}{80}$ & $0.00108253$ & $-$ & $0.00126295$ & $-$ & $0.00144338 $ & $-$ & $0.0016238$& $-$ 
& $0.000721688$ & $-$ & $0.000541266$ & $-$ & $0.000360844$ & $-$ & $0.000180422$& $-$ \\
\hline
\end{tabular}
\label{tab:ErrU82}}
\end{table}
\begin{table}
\caption{Pressure errors in $L^2$ norm for Algorithm \ref{thealgorithm4WK} with $J=8$ and $\gamma=1.5$}
\resizebox{\textwidth}{!}{
\centering
\small
\tabcolsep=0.01cm
\begin{tabular}{|c|c|c|c|c|c|c|c|c|c|c|c|c|c|c|c|c|c|c|c|c|c|c|c|}
\hline
 $\Delta t$ & $\|p_{1} - p_{1,h}\|$ & \text{rate} & $\|p_{2} - p_{2,h}\|$  & \text{rate} & $\|p_{3} - p_{3,h}\|$ & \text{rate} & $\|p_{4} - p_{4,h}\|$  & \text{rate} & $\|p_{5} - p_{5,h}\|$ & \text{rate} & $\|p_{6} - p_{6,h}\|$  & \text{rate} & $\|p_{7} - p_{7,h}\|$ & \text{rate} & $\|p_{8} - p_{8,h}\|$  & \text{rate}\\
\hline
$\frac{1}{10}$ & $0.0252714$ & $1.96492$ & $0.0290634 $ & $1.95369$ & $0.0327376 $ & $1.94240$ &  $0.0362957$ & $1.93106$
& $0.0173263$ & $1.98690$ & $0.0131794$ & $1.99809$ & $0.00891421  $ & $2.00969$ &  $0.00451913$ & $2.02070$   \\
\hline
$\frac{1}{20}$ & $0.00654547$ & $1.99176$ & $0.00761436 $ & $1.98956$ & $0.00867699 $ & $1.98735$ &  $0.00973338$ & $1.98514$
& $0.00438888$ & $1.99616$ & $0.00330115$ & $1.99836$ & $0.00220711  $ & $2.00056$ &  $0.00110675$ & $2.1523$   \\
\hline
$\frac{1}{40}$ & $0.00164993$ & $1.99797$ & $0.00192362 $ & $1.99736$ & $0.00219694 $ & $1.99696$ &  $0.0024699$ & $1.99645$
& $0.00110144$ & $1.99898$ & $0.000826642$ & $1.99950$ & $0.000551466  $ & $1.99999$ &  $0.000275922$ & $2.00053$   \\
\hline
$\frac{1}{80}$ & $0.00041332$ & $-$ & $0.000482175 $ & $-$ & $0.000550911 $ & $-$ &  $0.000619673$ & $-$
& $0.000275642$ & $-$ & $0.000206764$ & $-$ & $0.000137868   $ & $-$ &  $6.89438e-05$ & $-$   \\
\hline
\end{tabular}
\label{tab:ErrP82}}
\end{table}

\begin{table}
\caption{Velocity errors in $L^2$ norm for Algorithm \ref{thealgorithm4WK} with $J=8$ and $\gamma=2$}
\resizebox{\textwidth}{!}{
\centering
\small
\tabcolsep=0.01cm
\begin{tabular}{|c|c|c|c|c|c|c|c|c|c|c|c|c|c|c|c|c|c|c|c|c|c|c|c|}
\hline
 $\Delta t$ & $\|u_{1} - u_{1,h}\|$ & \text{rate} & $\|u_{2} - u_{2,h}\|$  & \text{rate} & $\|u_{3} - u_{3,h}\|$ & \text{rate} & $\|u_{4} - u_{4,h}\|$  & \text{rate} & $\|u_{5} - u_{5,h}\|$ & \text{rate} & $\|u_{6} - u_{6,h}\|$  & \text{rate} & $\|u_{7} - u_{7,h}\|$ & \text{rate} & $\|u_{8} - u_{8,h}\|$  & \text{rate}\\
\hline
$\frac{1}{10}$ & $0.0692549$ & $1.99963$ & $0.0807968 $ & $1.99962$ & $0.0923393 $ & $1.99962$ &  $0.103882$ & $1.99963$
& $0.0461727 $ & $1.99969$ & $0.034623$ & $1.99951$ & $0.0230711  $ & $1.99904$ &  $0.0115436$ & $1.99972$   \\
\hline
$\frac{1}{20}$ & $0.0173201$ & $1.99998$ & $0.0202068 $ & $1.99998$ & $0.0230935 $ & $1.99998$ &  $0.0259801$ & $1.99997$
& $0.0115467$ & $1.99998$ & $0.00866001$ & $1.99997$ & $0.00577333  $ & $1.99997$ &  $0.00288672$ & $1.99999$   \\
\hline
$\frac{1}{40}$ & $0.00433012$ & $2.0000$ & $0.00505181$ & $2.00000$ & $0.00577349$ & $1.99999$ & $0.00649518$& $1.99999$ 
& $0.00288674 $ & $1.99999$ & $0.00216506$ & $1.99999$ & $0.00144337$ & $1.99999$ & $0.000721684 $& $1.99999$ \\
\hline
$\frac{1}{80}$ & $0.00108253$ & $-$ & $0.00126295$ & $-$ & $0.00144338 $ & $-$ & $0.0016238$& $-$ 
& $0.000721688 $ & $-$ & $0.000541266$ & $-$ & $0.000360844$ & $-$ & $0.000180422$& $-$ \\
\hline
\end{tabular}
\label{tab:ErrU83}}
\end{table}

\begin{table}
\caption{Pressure errors in $L^2$ norm for Algorithm \ref{thealgorithm4WK} with $J=8$ and $\gamma=2$}
\resizebox{\textwidth}{!}{
\centering
\small
\tabcolsep=0.01cm
\begin{tabular}{|c|c|c|c|c|c|c|c|c|c|c|c|c|c|c|c|c|c|c|c|c|c|c|c|}
\hline
 $\Delta t$ & $\|p_{1} - p_{1,h}\|$ & \text{rate} & $\|p_{2} - p_{2,h}\|$  & \text{rate} & $\|p_{3} - p_{3,h}\|$ & \text{rate} & $\|p_{4} - p_{4,h}\|$  & \text{rate} & $\|p_{5} - p_{5,h}\|$ & \text{rate} & $\|p_{6} - p_{6,h}\|$  & \text{rate} & $\|p_{7} - p_{7,h}\|$ & \text{rate} & $\|p_{8} - p_{8,h}\|$  & \text{rate}\\
\hline
$\frac{1}{10}$ & $0.0252215$ & $1.96440$ & $0.0290096 $ & $1.95330$ & $0.0326806 $ & $1.94211$ &  $0.0362368$ & $1.93088$
& $0.0172889$ & $1.98620$ & $0.0131511$ & $1.99740$ & $0.00889508  $ & $2.00899$ &  $0.00450863$ & $2.01982$   \\
\hline
$\frac{1}{20}$ & $0.006536$ & $1.99074$ & $0.00760336 $ & $1.98855$ & $0.00866449$ & $1.98635$ &  $0.00971938$ & $1.98414$
& $0.00438251$ & $1.99514$ & $0.00329635$ & $1.99733$ & $0.0022039  $ & $1.99953$ &  $0.00110515$ & $2.00175$   \\
\hline
$\frac{1}{40}$ & $0.00164924$ & $1.99767$ & $0.0019228 $ & $1.99707$ & $0.002196 $ & $1.99666$ &  $0.00246885$ & $1.99612$
& $0.00110097$ & $1.99866$ & $0.00082629$ & $1.99919$ & $0.000551232  $ & $1.99970$ &  $0.000275806$ & $2.00021$   \\
\hline
$\frac{1}{80}$ & $0.000413272$ & $-$ & $0.000482113 $ & $-$ & $0.000550841 $ & $-$ &  $0.000619613$ & $-$
& $0.000275612$ & $-$ & $0.000206739$ & $-$ & $0.00013785$ & $-$ &  $6.89367e-05$ & $-$   \\
\hline
\end{tabular}
\label{tab:ErrP83}}
\end{table}
\subsection{Flow around a cylinder} 
We test our Algorithms \ref{thealgorithm5WK}-\ref{thealgorithm6WK} on a two dimensional channel flow around a cylinder, a well-known benchmark problem taken
from Sh\"{a}fer and Turek \cite{ST96}. The flow patterns are driven by the interaction of a fluid with a wall which is an important scenario for many industrial flows. The domain for the problem is a $2.2\times0.41$ rectangular channel with a cylinder of radius 0.05 centered at $(0.2, 0.2)$ (taking the bottom left corner of the rectangle as the origin). The cylinder, top and bottom of the channel are prescribed no slip boundary conditions, and the time dependent inflow and outflow profile are
\begin{align*}
u_1(0,y,t)&=u_1(2.2,y,t)=\frac{6}{0.41^2}\sin(\pi t/8)y(0.41-y) \\
u_2(0,y,t)&=u_2(2.2,y,t)=0
\end{align*}
The quantative results for this problem with $\nu=\frac{1}{1000}$ are given in \cite{John2004} and \cite{LIU20097250} under Dirichlet outflow and 
do-nothing outflow conditions, respectively. Here we chose two different number of ensemble members, namely three ensemble members $J=3$ and seven ensemble 
members $J=7$. For the three ensemble members $J=3$, we have viscosities as $\nu_1=\frac{1}{1000}$, $\nu_2=\frac{1}{900}$, $\nu_3=\frac{1}{800}$.
For the seven ensemble members $J=7$, viscosities are  $\nu_1=\frac{1}{1000}$, $\nu_2=\frac{1}{900}$, $\nu_3=\frac{1}{800}$,  $\nu_4=\frac{1}{700}$,
$\nu_5=\frac{1}{1100}$, $\nu_6=\frac{1}{1200}$ and $\nu_7=\frac{1}{1300}$. We compare the results of the $\nu_1$ case with the reference values. 
The mesh used in the simulations is shown in Fig. \ref{fig:mesh} with diameter $h=0.0216741$. The smallest eigenvalue of the Dirichet-Neumann problem
\eqref{eq:eigenvalue} is computed to be $\lambda_1 = 59.3467$. 
\begin{figure}
\centering
\includegraphics[width=12cm, trim={2.9cm 2.3cm 0 27cm},clip]{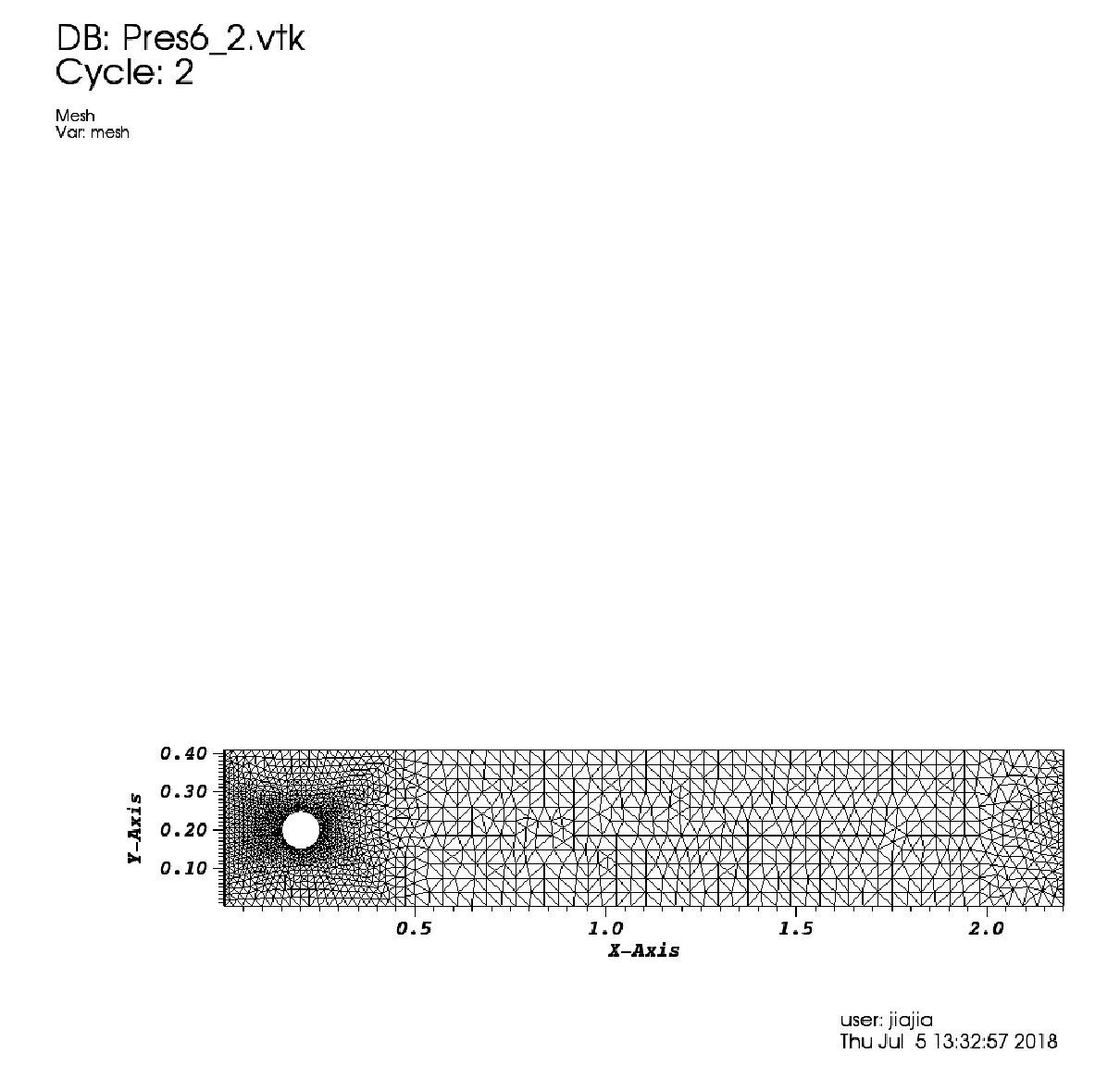}
\caption{The finite element mesh used in flow around a cylinder experiment. Number of elements is 3306 }
\label{fig:mesh}       
\end{figure}
We started all cases with the $\Delta t=0.004$. Stability is checked according to the inequalities 
\eqref{eq:StabConditionOBC1_2nd}-\eqref{eq:StabConditionOBC2_2nd} for the Algorithm \ref{thealgorithm5WK} with $L=0.01$, and
\eqref{eq:StabConditionOBC3_2nd} for the Algorithm \ref{thealgorithm6WK}. 
Since $\frac{\nu_\infty}{\nu_0}=1.25$ for $J=3$ and $\frac{\nu_\infty}{\nu_0}=1.7143$ for $J=7$, 
we choose $\gamma=1.5$ and $\sigma=0.75$ in all cases. 
If the stability inequalities are violated, the time step is halved. The final value of $\Delta t $ was $0.001$ 
for $J=3$ case and $6.25e-5$ for $J=7$ case, for both Algorithms. The time step change history is shown in table \ref{tab:time1}, w
here you can see that the Algorithm \ref{thealgorithm6WK} refines the time step earlier than the Algorithm \ref{thealgorithm5WK},
and thus requires slightly longer execution time. 
\begin{table}
\centering
\caption{time step change history for Algorithm \ref{thealgorithm5WK} and \ref{thealgorithm6WK} with three ensembles $J=3$ and seven ensembles $J=7$}
\begin{tabular}{|c|c|c|c|c|}
\multirow{2}{*}{$\Delta t$} & \multicolumn{2}{c|}{three ensembles} & \multicolumn{2}{c|}{seven ensembles}\\
\cline{2-5}
&$\text{Algorithm \ref{thealgorithm5WK}}$& $\text{Algorithm \ref{thealgorithm6WK}} $& $\text{Algorithm \ref{thealgorithm5WK}}$& $\text{Algorithm \ref{thealgorithm6WK} }$\\
\cline{1-5}
 $0.002$ & $t=4.532$ & $t=4.568$ & $t=3.684$  & $t=3.664$\\
\hline
$0.001$ & $t=4.694$ & $t=4.718$ & $t=3.752$ & $t=3.674$ \\
\hline
$0.0005$ & -- & -- & $t=3.796$  &$t=3.688$ \\
\hline
$0.00025$ & -- & -- & $t=3.9235$  & $t=3.8125$\\
\hline
$0.000125$ & -- & -- & $t=4.1425$  & $t=4.0225$\\
\hline
$6.25e-5$ & --  & -- & $t=4.55012$  & $t=4.0265$\\
\hline
\end{tabular}
\label{tab:time1}
\end{table}
We compute values for the maximal drag $c_{d,max}$ and lift $c_{l,max}$ coefficients on the cylinder boundary, and 
the pressure difference $\Delta p(t)$ between the front and back of the cylinder at the final time $T = 8$. 
The time evolutions of the these quantities are given in Fig. \ref{fig:dlp}, and in both cases we get almost identical results. 

The maximum lift and drag coefficients and pressure drop for the simulations are given in Table \ref{tab:dlp}.
In general, we observe results that are close to the reference values. In particular, as $J$ grows, the errors also seem to increase, as one would expect from 
the convergence Theorem \ref{convthm}.
The velocity contour plots at times t = 6, 8 are presented in Figs. \ref{fig:5v} and \ref{fig:6v}, and streamlines are given in those plots to show a vortex street.
Qualitatively, the plots match the reference plots from \cite{LIU20097250}, 
and those two algorithms \ref{thealgorithm5WK} and \ref{thealgorithm6WK} gave the same results. 
We compare Figs. \ref{fig:5v} and \ref{fig:6v} with the results obtained using the open boundary or zero traction boundary conditions in \cite{LIU20097250}. With our method, at $t=8$ the last eddy 
is cut through by the outflow boundary $x=2.2$. This agrees with the results in \cite{LIU20097250}
unlike giving a prescribed parabolic velocity profile where the last eddy will remain on the left hand side of $x=2.2$ completely, as in \cite{John2004}. The prescribed Dirichlet type parabolic outflow profile is less physical because following the previous alternating pattern from upstream, it is unrealistic that both eddies near the top and bottom walls will vanish at the same position at $x=2.2$. 
\begin{figure}
\centering
  \includegraphics[width = 1.4in, height = 1.6in]{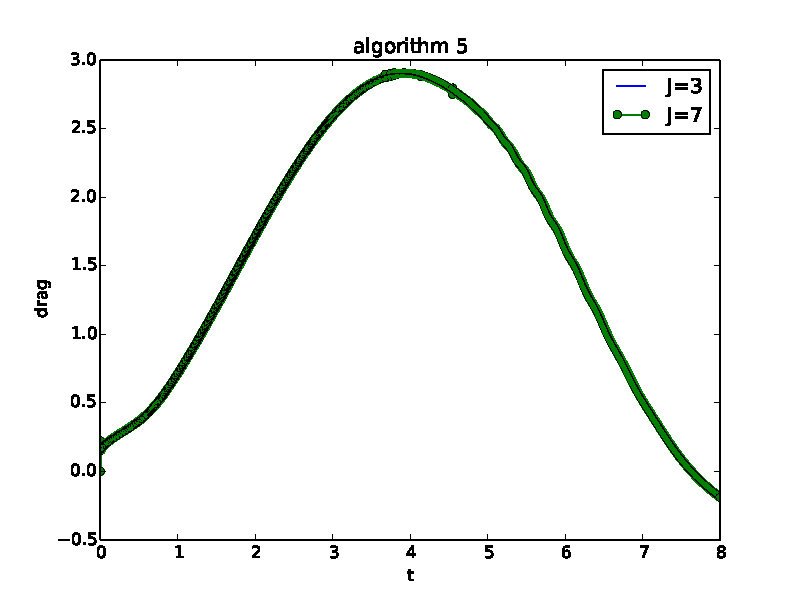}
  \includegraphics[width = 1.4in, height = 1.6in]{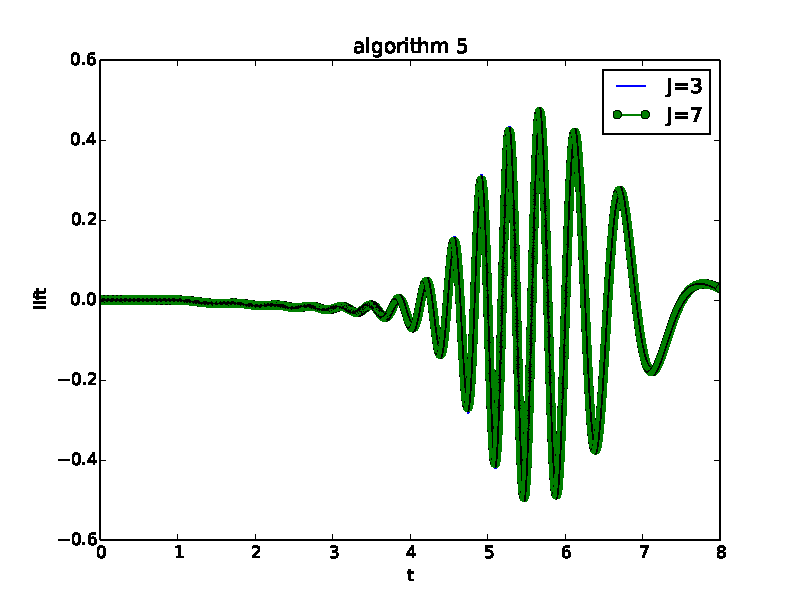}
  \includegraphics[width = 1.4in, height = 1.6in]{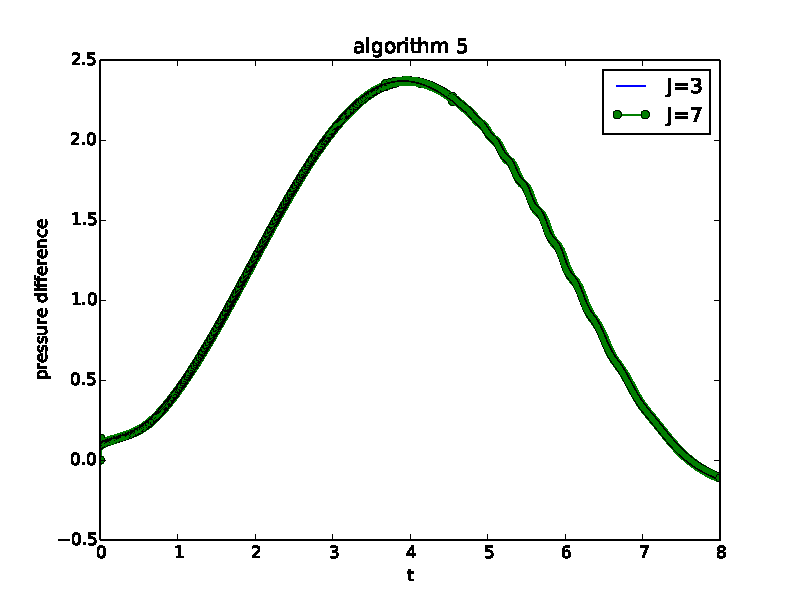}
\\  
  \includegraphics[width = 1.4in, height = 1.6in]{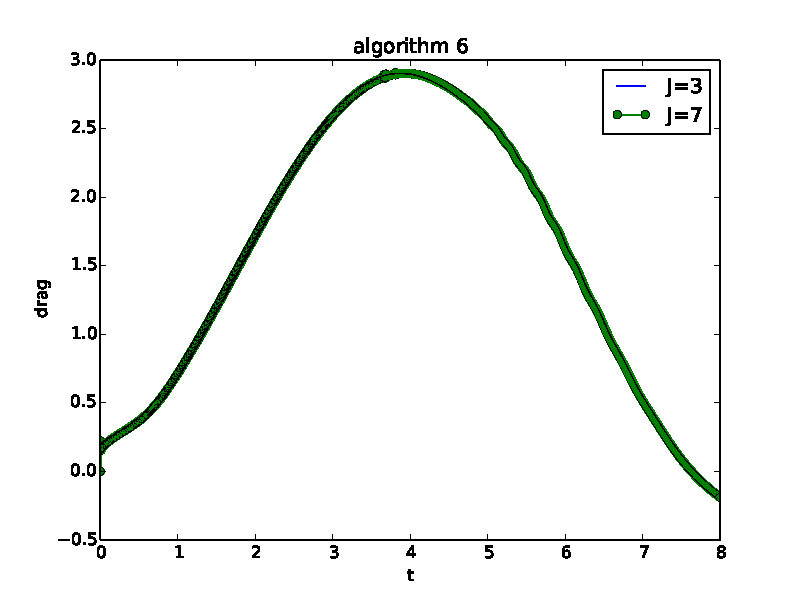}
  \includegraphics[width = 1.4in, height = 1.6in]{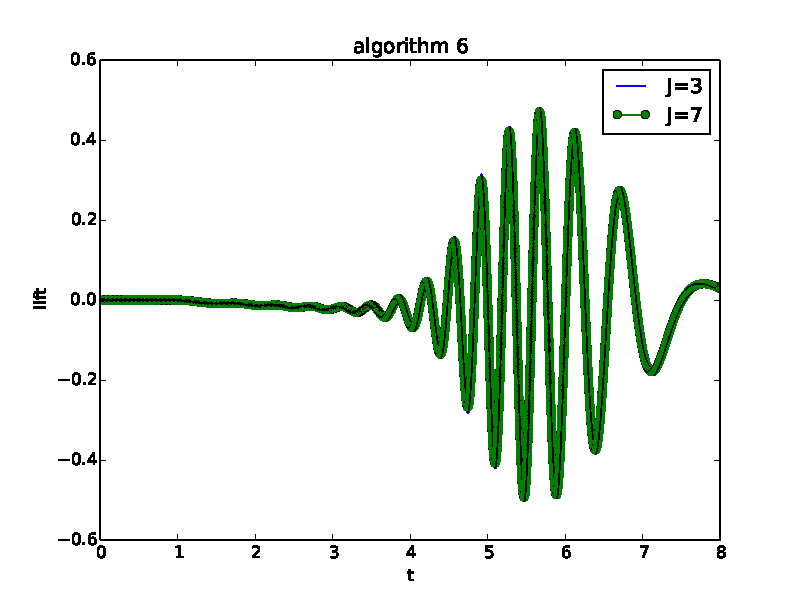}
  \includegraphics[width = 1.4in, height = 1.6in]{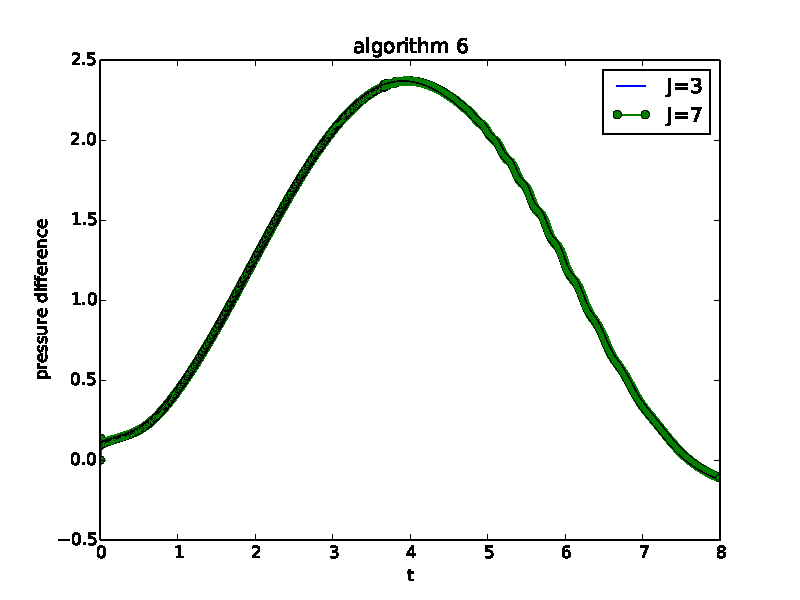}
\caption{From left to right: the drag and lift coefficients $c_d$, $c_l$ and pressure difference between front and back of the cylinder $\Delta p$ for
flow past a cylinder with Algorithm \ref{thealgorithm5WK} (top) and Algorithm \ref{thealgorithm6WK} (bottom) for both $J=3$ and $J=7$. }
\label{fig:dlp}       
\end{figure}
\begin{table}
\centering
\caption{Drag, Lift and pressure drop values}
\begin{tabular}{c|c|c|c|c|c|c}
\multicolumn{2}{c|}{ $\text{method} $}& $c_{d,max}$ & $t(c_{d,max})$ & $c_{l,max}$ & $t(c_{l,max})$ & $\Delta p$  \\
\hline
\multirow{2}{*}{$\text{Algorithm \ref{thealgorithm5WK}}$} &$ J=3$ & $2.90226$ & $3.936$ & $0.477011$ & $5.677$  &$-0.112623$ \\
\cline{2-7}
&$ J=7$ & $2.90226$ & $3.7965$ & $0.470657$ & $5.67288$  &$-0.112367$  \\
\hline
\multirow{2}{*}{$\text{Algorithm \ref{thealgorithm6WK}}$} &$ J=3$ & $2.90226$ & $3.936$ & $0.477247$ & $5.677$  &$-0.112632$ \\
\cline{2-7}
&$ J=7$  & $2.90226$ & $3.81275$ & $0.470393$ & $5.67275$  &$-0.112343$ \\
\hline
\multicolumn{2}{c|}{ $\text{(Dirichlet)} \cite{John2004} $} & $2.95092$  & $3.93625 $ & $0.47795 $ & $5.69313$  & $-0.1116$ \\
\hline
\multicolumn{2}{c|}{ $\text{(No-traction)} \cite{LIU20097250} $} & $2.9513$  & $4.0112$ & $0.47887 $ & $5.6928$  & $-0.026382$ \\
\hline
\end{tabular}
\label{tab:dlp}
\end{table}
%
%
\begin{figure}
\centering
  \includegraphics[scale=0.2]{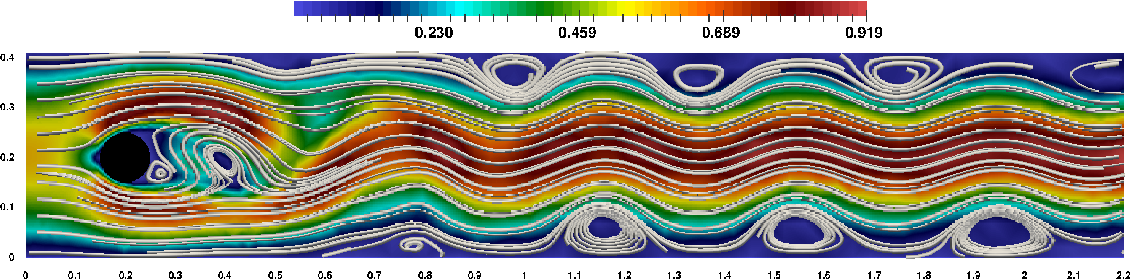}
  \includegraphics[scale=0.2]{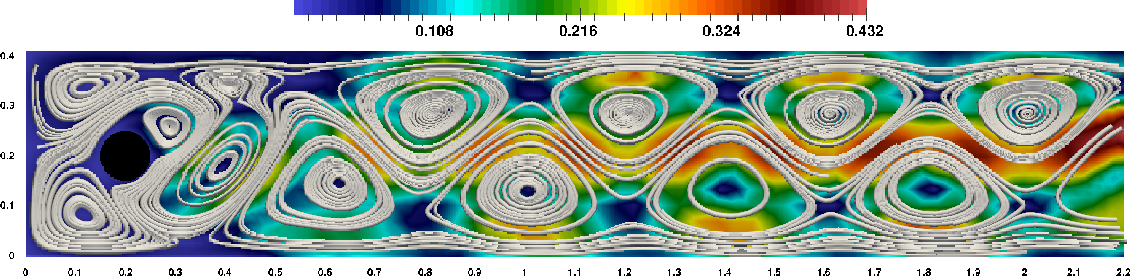} \\
  \includegraphics[scale=0.2]{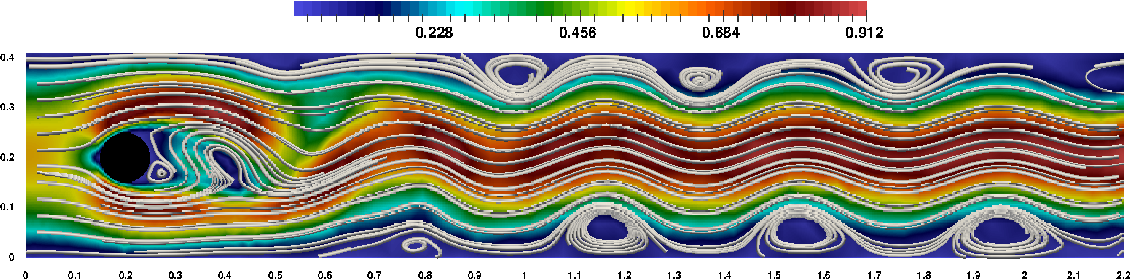}
  \includegraphics[scale=0.2]{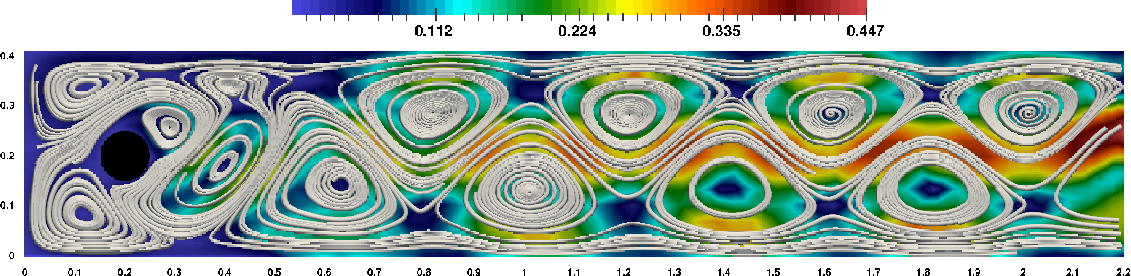}
\caption{Speed contours at t = 6,8, with Algorithm \ref{thealgorithm5WK} for $J=3$ (top) and $J=7$ (bottom)}
\label{fig:5v}       
\end{figure}
\begin{figure}
\centering
  \includegraphics[scale=0.2]{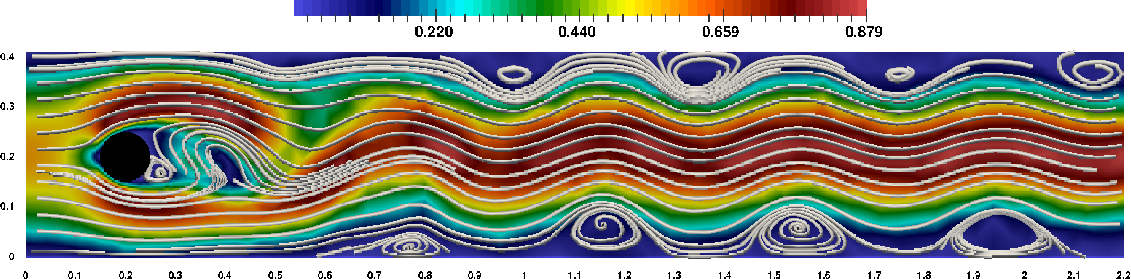}
  \includegraphics[scale=0.2]{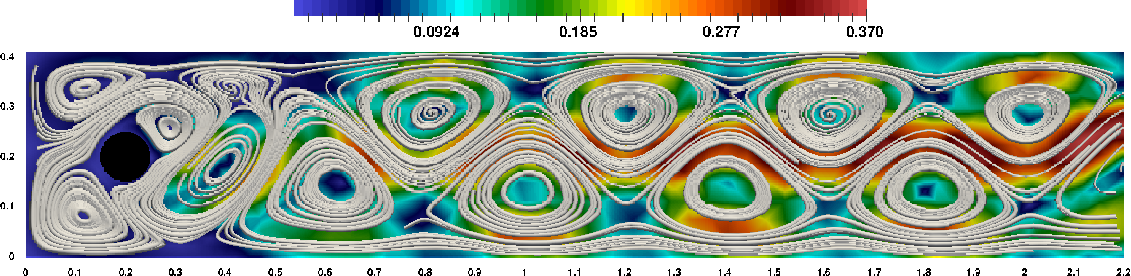}
  \includegraphics[scale=0.2]{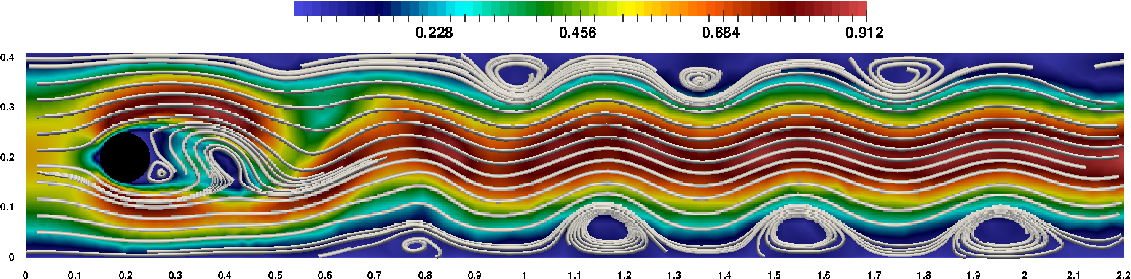}
  \includegraphics[trim={0.05in 0.7in 0in 60in},scale=0.2001]{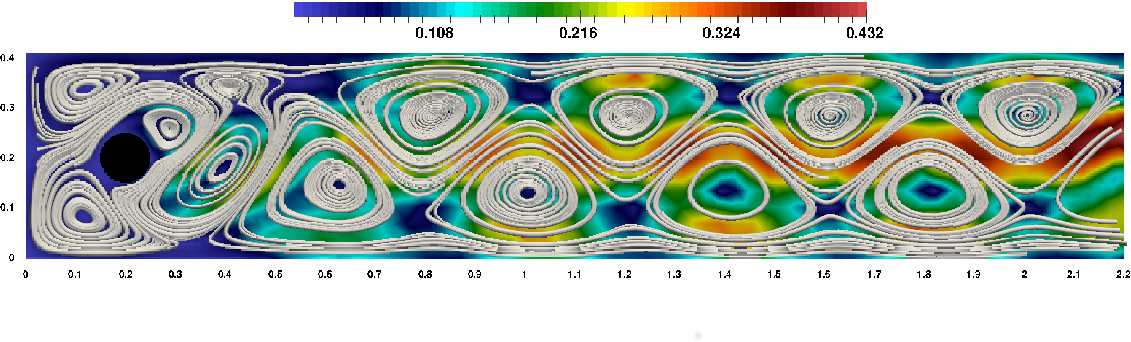}
\caption{Speed contours at t = 6,8, with Algorithm \ref{thealgorithm6WK} for $J=3$ (top) and $J=7$ (bottom)}
\label{fig:6v}       
\end{figure}
\subsection{Channel flow with a contraction and two outlets}
Our last experiment is for a complex 2-d flow through 
a channel with a contraction and two outlets, one on the top of the channel and the other one is at the end of the channel.1
Mesh is shown in Fig. \ref{fig:mesh1}
\begin{figure}[h!]
\centering
  \includegraphics[width=10cm, trim={3.1cm 2.3cm 1cm 27cm},clip]{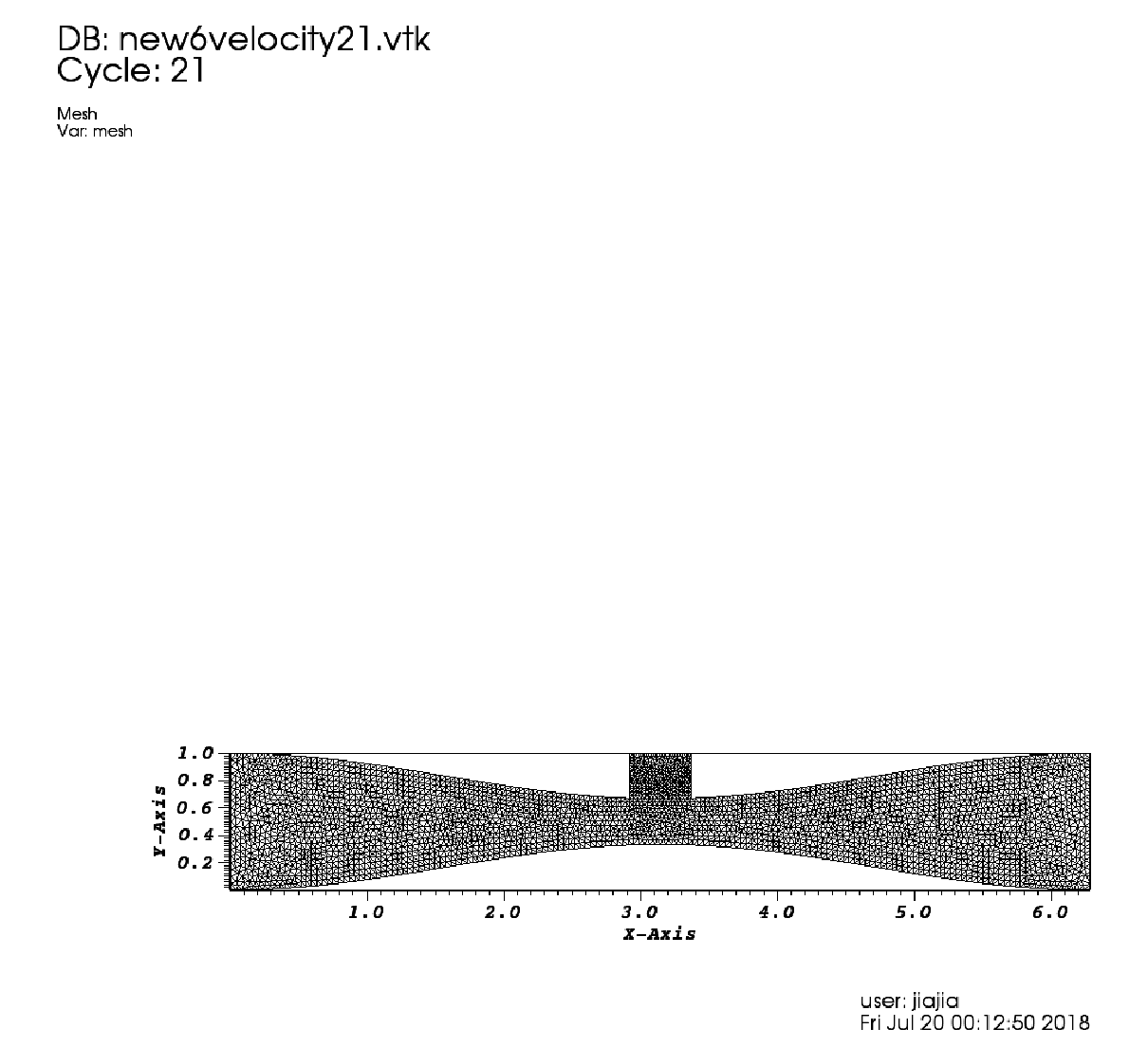}
\caption{The finite element mesh used in Channel flow with a contraction and two outlets. DOF is 16,672.}
\label{fig:mesh1}       
\end{figure}

Here we restrict ourselves to the $J=3$ case. We run the simulations on time interval $(0,4)$, with 
$\nu_1 = 0.001$, $\nu_2 = 0.003$ and $\nu_3 = 0.005$, $(X_h,Q_h)=(P_2,P_1)$ Taylor-Hood finite element pair. 
The velocity boundary conditions are: no-slip on the walls, $g_1 = (4y(1-y),0)^{T}$, $g_2 = (1+\varepsilon)(4y(1-y),0)^{T}$, 
$g_3 = (1-\varepsilon)(4y(1-y),0)^{T}$
at the inlet, and open boundary condition at the outlets.
Here we test our Algorithm \ref{thealgorithm5WK} with $\LL=0.01$ and Algorithm \ref{thealgorithm6WK}.
Initial conditions $u^0_{1,h}, u^0_{2,h}, u^0_{3,h}$ are obtained by solving Stokes equations in the same domain with perturbed body forces
$f_1 = \varepsilon(0,0)^T$,  $f_2 = \varepsilon(\cos(\pi x y+t),\sin(\pi (x+y)+ t))^T$  and $f_3 = \varepsilon(\sin (\pi (x+y)+ t),\cos(\pi xy+t))^T$
with $\varepsilon = 10^{-2}$.

We ran both ensemble and independent runs.  In the ensemble runs, we start with $\Delta t = 0.01$, and half the time step once the 
stability condition inequalities is violated. For Algorithm \ref{thealgorithm5WK}, inequalities \eqref{eq:StabConditionOBC1_2nd}-\eqref{eq:StabConditionOBC2_2nd}
are checked  with $\lambda_1=11.8335$, $\gamma=1$ and $\sigma=0.961538$ and the final time step is $\Delta t = 3.90625e-5$. 
For Algorithm \ref{thealgorithm6WK}, stability  condition inequality \eqref{eq:StabConditionOBC3_2nd} is used with $C=1$, $\gamma=1$ and $\sigma=0.961538$,
and the final time step is $\Delta t = 1.953125e-5$. In order to find the optimal $C$ in the inequality \eqref{eq:StabConditionOBC3_2nd},
a few precalculations were performed. Nonetheless, the Algorithm \ref{thealgorithm6WK} requires more execution time
than the Algorithm \ref{thealgorithm5WK}.

Simulations are performed on a mesh with $16,672$ DOF. Since the simulations on this mesh are underresolved, we use the adaptive nonlinear
filter scheme of \cite{Takhirov2018} to stabilize the solutions. As a reference, we also performed independent DNS runs for $\nu_2$ and $\nu_3$ on a mesh with
total of $290,000$ DOF and second order timestepping scheme.
Due to the computational cost, we only ran the DNS simulations till $T=1$ and compare speed contours with Algorithm
\ref{thealgorithm5WK} in  Fig. \ref{fig:2dCar_ens1} and Algorithm \ref{thealgorithm6WK} in Fig. \ref{fig:2dCar_ens2}.
We can observe that, the ensemble scheme gives qualitatively same results as independent simulations. 
The speed contour for $\nu_1 = 0.001$ is shown for ensemble method and independent runs.
Notice that in the ensemble runs, no perturbation was added in the $\nu_1 = 0.001$ case, for fair comparison. 
The Fig. \ref{fig:2dCar_ens3} corresponds to the Algorithm \ref{thealgorithm5WK}as well,
and Fig. \ref{fig:2dCar_ens4} is with Algorithm \ref{thealgorithm6WK}, and they gave very similar results. 
\begin{figure}
\centering
\includegraphics[scale=0.18]{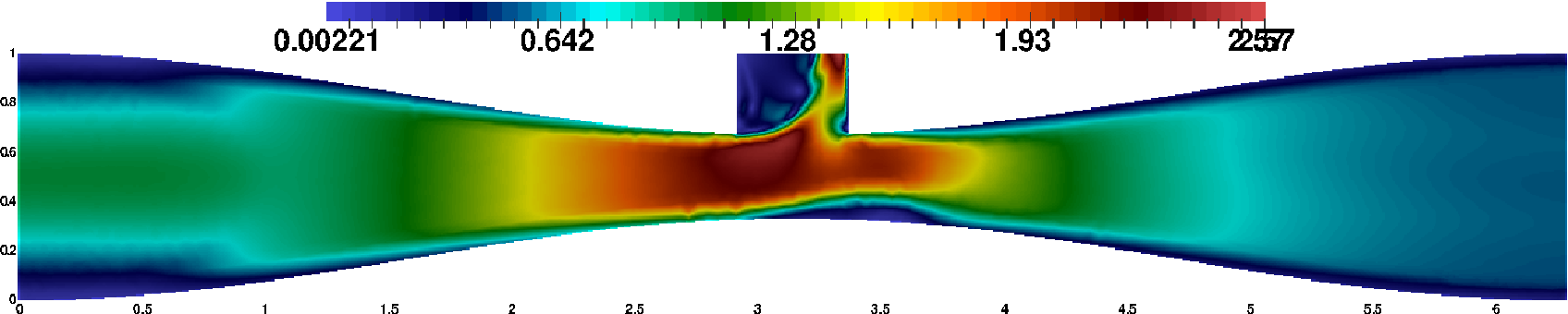}\\
\includegraphics[scale=0.38]{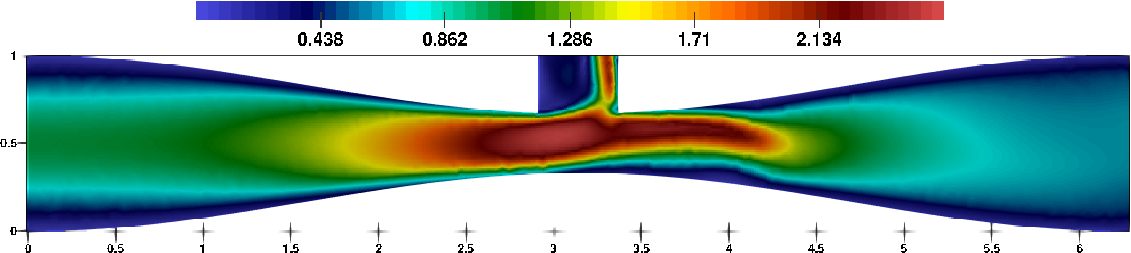}\\
\includegraphics[scale=0.18]{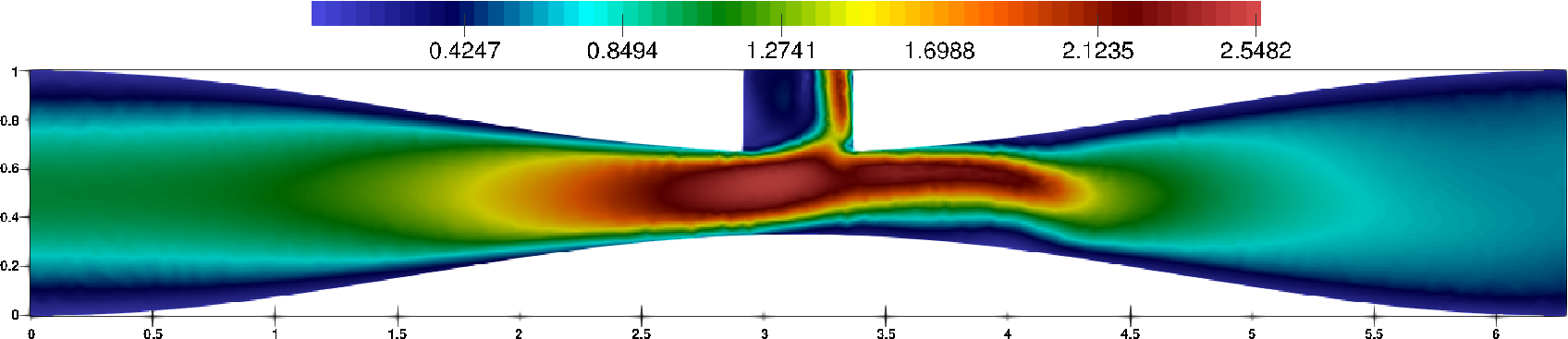}
\caption{Plots of speed contour for $\nu_2$  from DNS, ensemble simulation and independent simulation from top to bottom at $T = 1$ with Algorithm \ref{thealgorithm5WK} .} 
\label{fig:2dCar_ens1}
\end{figure}
\begin{figure}
\centering
\includegraphics[scale=0.18]{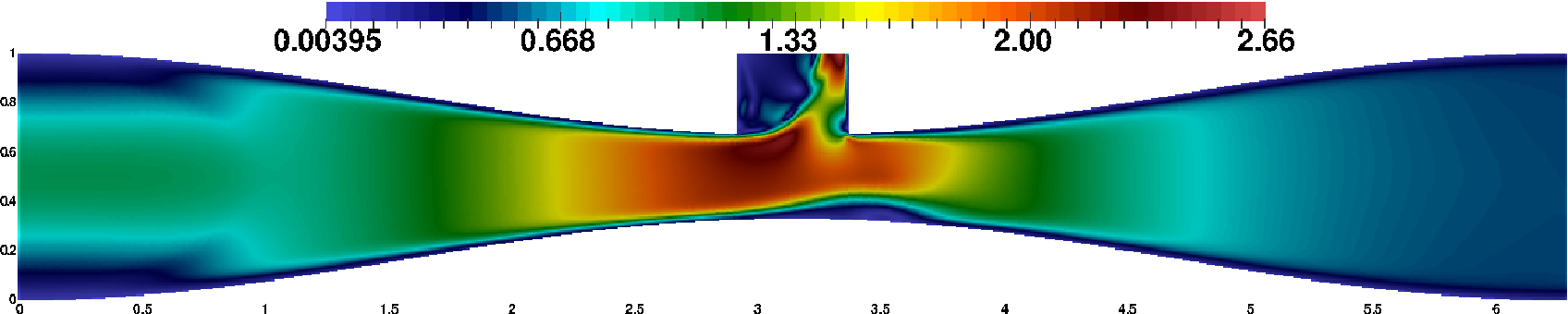}\\
\includegraphics[scale=0.38]{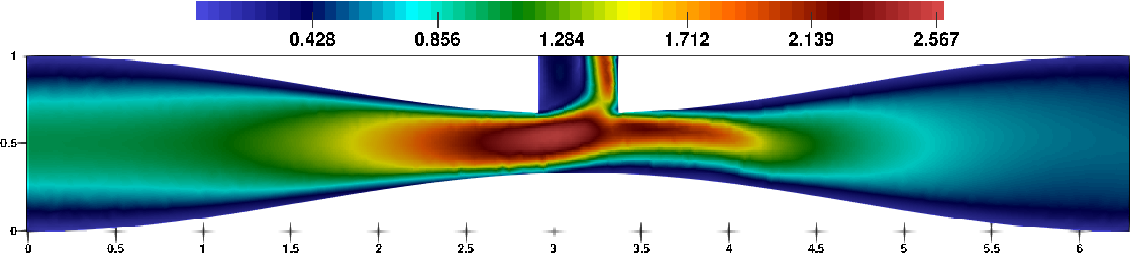}\\
\includegraphics[scale=0.18]{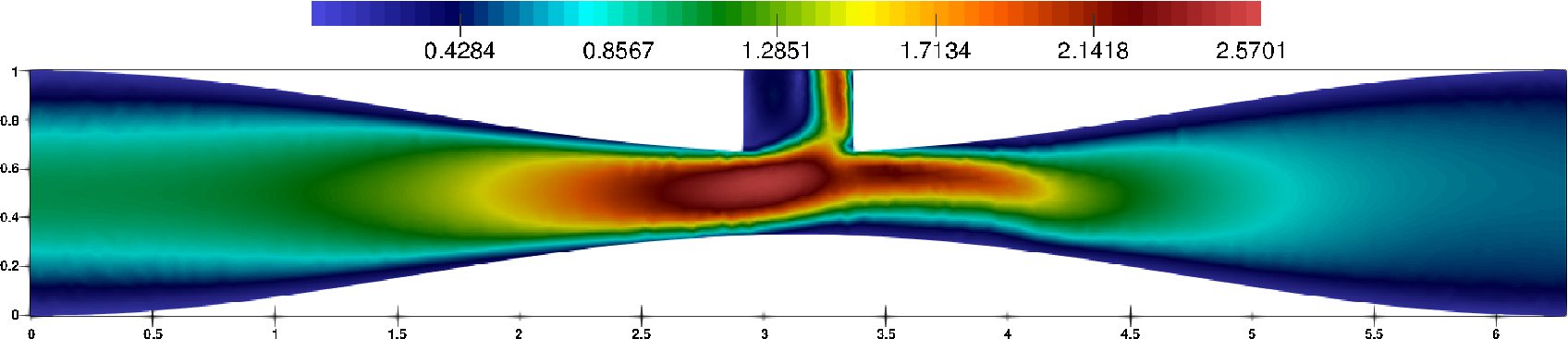}
\caption{Plots of speed contour for $\nu_3$  from DNS, ensemble simulation and independent simulation from top to bottom at $T = 1$ with Algorithm \ref{thealgorithm5WK} .}
\label{fig:2dCar_ens2}
\end{figure}
\begin{figure}
\centering
\includegraphics[scale=0.18]{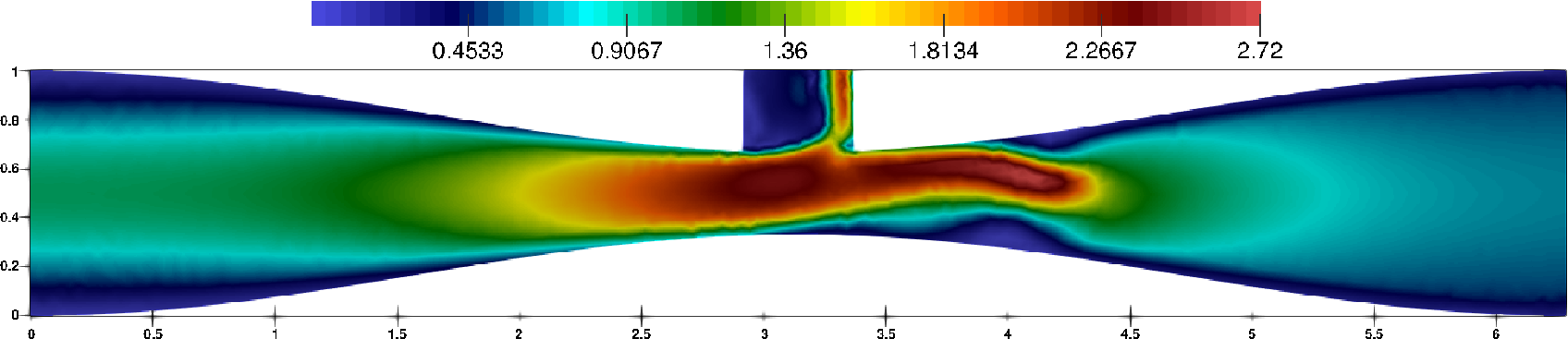}\\
\includegraphics[scale=0.38]{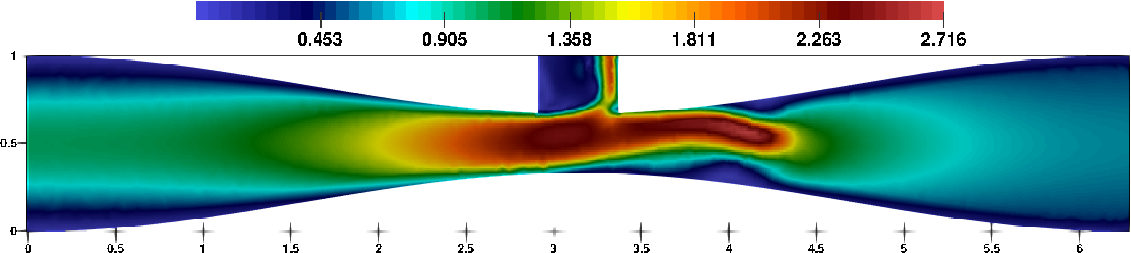}\\
\includegraphics[scale=0.18]{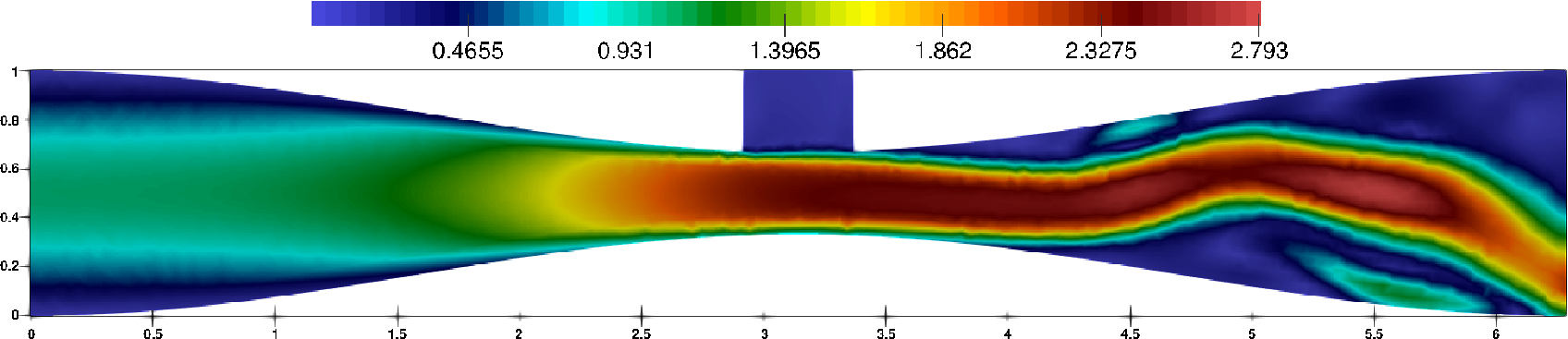}\\
\includegraphics[scale=0.38]{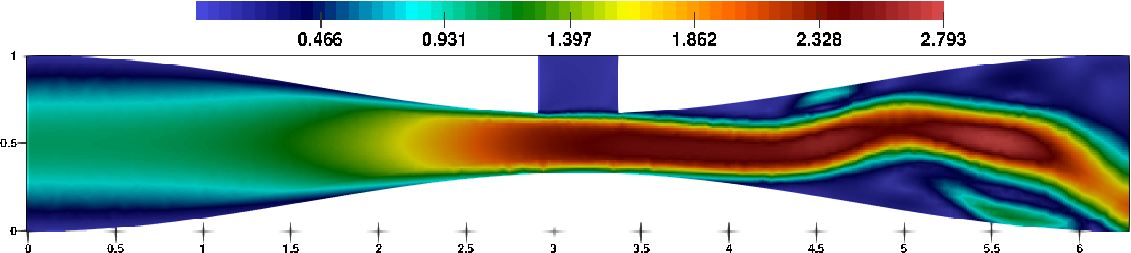}
\caption{Plots of speed contour for $\nu_1$ with Algorithm \ref{thealgorithm5WK} simulation and independent simulation from top to bottom at $T = 1$ (top 2) 
and $T = 4$ (bottom 2).}
\label{fig:2dCar_ens3}
\end{figure}

\begin{figure}
\centering
\includegraphics[scale=0.18]{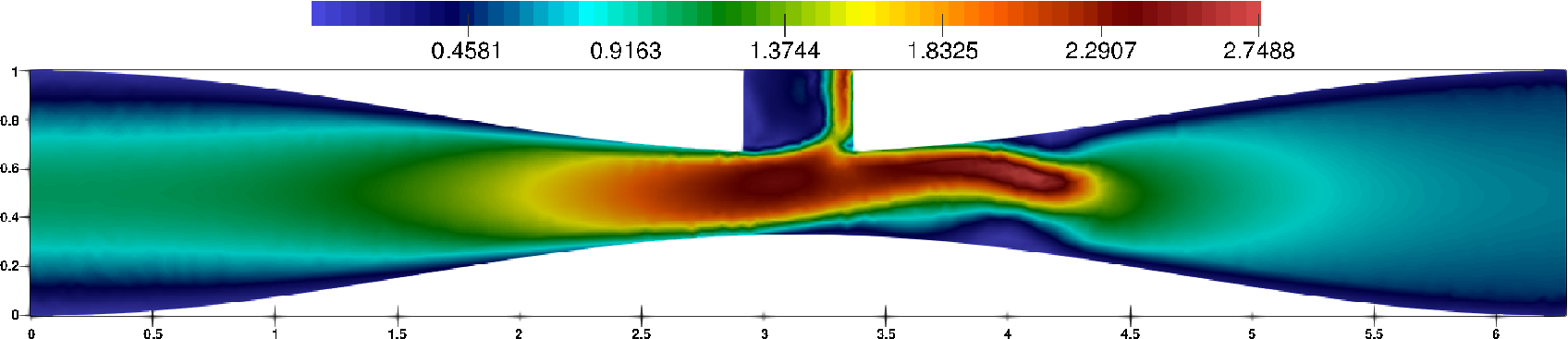}\\
\includegraphics[scale=0.18]{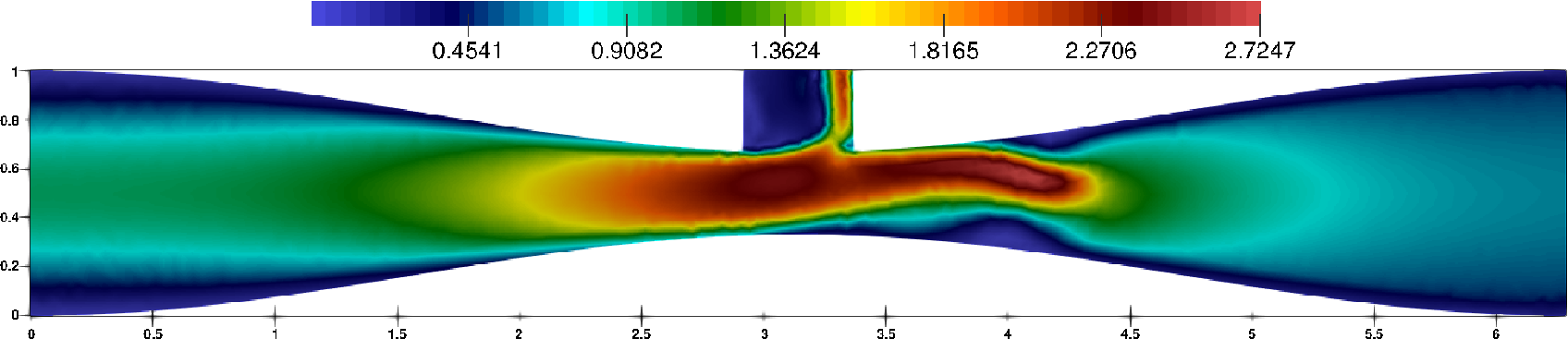}\\
\includegraphics[scale=0.18]{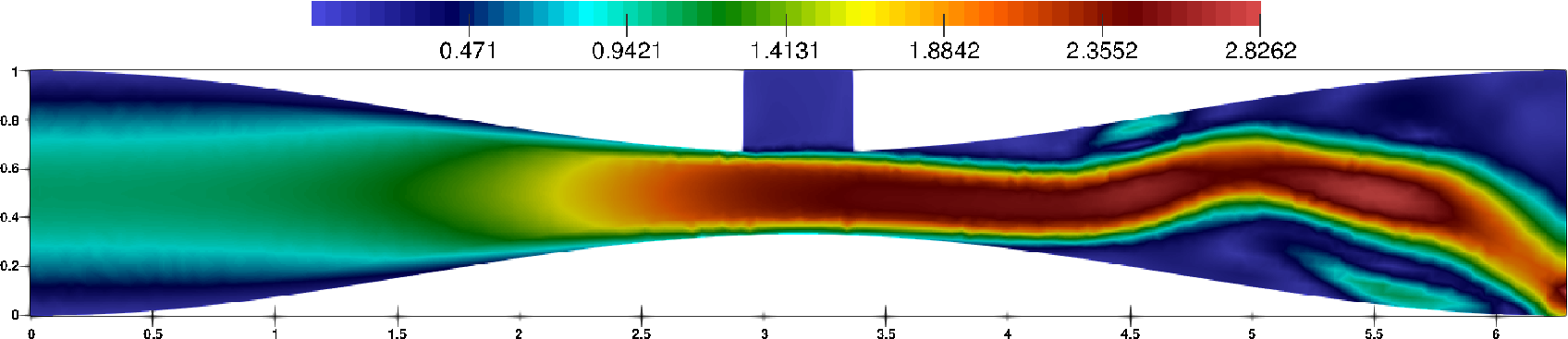}\\
\includegraphics[scale=0.18]{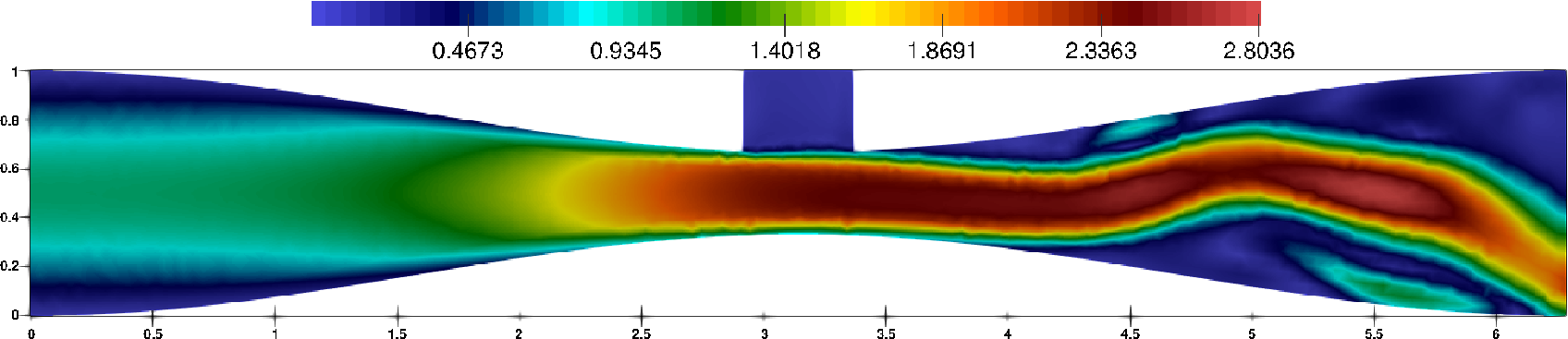}
\caption{Plots of speed contour for $\nu_1$  from Algorithm \ref{thealgorithm6WK} simulation and independent simulation from top to bottom at $T = 1$ (top 2) and $T = 4$
(bottom 2).}
\label{fig:2dCar_ens4}
\end{figure}
\section{Conclusions}
\label{sec:Conclusion}
We revisited the algorithm of \cite{doi:10.1093/imanum/dry029}, 
and proposed a new one with better stability properties. We also developed first and second order ensemble schemes for 
open boundary conditions, with provable stability bounds. While the Algorithm \ref{thealgorithm5WK} has pessimistic dependance on the 
final time $T$, it proved to be far more robust than the Algorithm \ref{thealgorithm6WK}.
The numerical tests at moderate $\Reynolds$ number show that the ensemble simulation match the independent simulation results, both 
qualitatively and quantatively. 

We believe that the ensemble schemes need further testing and research in order to fully understand their advantages and possible disadvantages.
One of the projects we will undertake in the future is the CPU time comparison of an ensemble scheme against fully explicit scheme on adaptively refined meshes 
for high $\Reynolds$ number flows.



\begin{thebibliography}{}

 \bibitem{doi:10.1002/fld.1168} Pahlevani, Faranak. \emph{Sensitivity computations of eddy viscosity models with an application in drag 
computation}, Int. J. Num. Meth. Fluids, 52(4), pp. 381-392, 2016.
 
 \bibitem{BURKARDT2006337} Burkardt, John and Gunzburger, Max and Lee, Hyung-Chun. \emph{POD and CVT-based reduced-order modeling of 
 Navier?Stokes flows}, Comp. Meth. Appl. Mech. Eng., 196(1), pp. 337-355, 2006.
 
 \bibitem{HOWARD2017333} Howard, Clint and Gupta, Sushen and Abbas, Ali and A.G. Langrish, Timothy and F. Fletcher, David. 
\emph{Proper Orthogonal Decomposition (POD) analysis of CFD data for flow in an axisymmetric sudden expansion}, Chem. Eng. Res. Des,
 123, pp. 333 - 346, 2017.
 
 \bibitem{WALTON20138930} S. Walton and O. Hassan and K. Morgan. \emph{Reduced order modelling for unsteady fluid flow using proper 
 orthogonal decomposition and radial basis functions}, Appl. Math. Model., 37(20), pp. 8930 - 8945, 2013.

 \bibitem{doi:10.1175/JAS-D-14-0250.1} Christensen, H. M. and Moroz, I. M. and Palmer, T. N. \emph{Stochastic and Perturbed Parameter
 Representations of Model Uncertainty in Convection Parameterization}, 72(6), pp. 2525-2544, 2015.
 
 \bibitem{doi:10.1175/1520-0493(1997)125<3297:EFANAT>2.0.CO;2} Toth, Zoltan and Kalnay, Eugenia. \emph{Ensemble Forecasting at NCEP and 
 the Breeding Method}, Monthly Weather Review, 125(12), pp. 3297-3319, 1997.

 \bibitem{Nan2ndOrderVisc} M. Gunzburger, N. Jiang and Z. Wang, \emph{A Second-Order Time-Stepping Scheme for Simulating Ensembles of Parameterized Flow Problems}, 
 Comput. Methods Appl. Math.,  2017.
 
 \bibitem{doi:10.1093/imanum/dry029}
 Gunzburger, Max and Jiang, Nan and Wang, Zhu.
 \emph{An efficient algorithm for simulating ensembles of parameterized flow problems}, 
 IMA Journal of Numerical Analysis, 2018.
 
 \bibitem{Jiang_2014}
 Nan Jiang and William Layton.
 \emph{An algorithm for fast calculation of flow ensembles}.
 International Journal for Uncertainty Quantification, 2152-5080, 4 (4), 273--301, 2014.
 
 \bibitem{doi:10.1002/fld.1650181006}  Sani, R. L. and Gresho, P. M., \emph{Résumé and remarks on the open boundary condition minisymposium},
 Int. J. Num. Meth. Fluids, 983-1008, 10 (18), 1994.
 
 \bibitem{FLD:FLD307} Heywood, J. G. and Rannacher, R. and Turek, S., \emph{Artificial boundaries and flux and pressure conditions for the incompressible
 {Navier}-{Stokes} equations}, Int. J. Numer. Meth. Fluids, 325-352, 22(5), 1996.

 \bibitem{doi:10.1002/cnm.2918} Bertoglio, Crist\'{o}bal and Caiazzo, Alfonso and Bazilevs, Yuri and Braack, Malte and Esmaily, Mahdi and Gravemeier, 
 Volker and L. Marsden, Alison and Pironneau, Olivier and E. Vignon-Clementel, Irene and A. Wall, Wolfgang, \emph{Benchmark problems for numerical treatment 
 of backflow at open boundaries}, Int. J. Num. Meth. Biomed. Engrg., e2918, 34(2), 2018. 
 
 \bibitem{JIANG2016388} \emph{An optimally accurate discrete regularization for second order timestepping methods for Navier–Stokes equations},
Comp. Meth. Appl. Mech. Engrg.",
volume = "310",
pages = "388 - 405",
year = "2016",
issn = "0045-7825",
doi = "https://doi.org/10.1016/j.cma.2016.07.017",
url = "http://www.sciencedirect.com/science/article/pii/S0045782516307587",
author = "Nan Jiang and Muhammad Mohebujjaman and Leo G. Rebholz and Catalin Trenchea",
 
 \bibitem{John2004}
John, V.,
\emph{Reference values for drag and lift of a two dimensional time-dependent flow around a
cylinder},
Int. J. Numer. Meth. Fluids, 44, pp. 777-788, 2004

 \bibitem{LIU20097250}
 Jie Liu. \emph{Open and traction boundary conditions for the incompressible Navier-Stokes equations},
J. Comp. Phys., 228(19), pp. 7250 - 7267, 2009.

 \bibitem{NUM:NUM21908}
Jiang, Nan and Layton, William.
\emph{Numerical analysis of two ensemble eddy viscosity numerical regularizations of fluid motion},
Num. Meth. Part. Diff. Eq., 31(3), pp. 630-651, 2015.

 \bibitem{Jiang2015}
Jiang and Nan, \emph{A Higher Order Ensemble Simulation Algorithm for Fluid Flows}, J. Sci. Comp.,
64(1), pp. 264-288, 2015.

 \bibitem{NUM:NUM22024}
Takhirov, Aziz and Neda, Monika and Waters, Jiajia.
\emph{Time relaxation algorithm for flow ensembles},
Num. Met. Part. Diff. Eq., 32(3), pp. 757-777, 2016.

\bibitem{DONG2015300}
S. Dong,
\emph{A convective-like energy-stable open boundary condition for simulations of incompressible flows},
J. Comp. Phys., 302, pp. 300-328, 2015.

\bibitem{Quar09}
Alfio Quarteroni, 
\emph{Numerical Models for Differential Problems},
Springer-Verlag Milan, 2009.

\bibitem{MR3043640}
Hecht, F., 
\emph{New development in FreeFem++},
J. Numer. Math., 20(3-4), pp 251-265, 2012.

\bibitem{ST96}
M. Sch$\ddot{\mbox{a}}$fer and S. Turek,
\emph{The benchmark problem `flow around a cylinder' flow simulation with high performance computers II}, 
in E.H. Hirschel (Ed.), Notes on Numerical Fluid Mechanics,
52, pp. 547-566, 1996.

\bibitem{JBILOU199997}
K. Jbilou,
\emph{Smoothing iterative block methods for linear systems with multiple right-hand sides},
J. Comp. Appl. Math., 107(1), pp. 97-109, 1999.

\bibitem{Heyouni2005}
Heyouni, M. and Essai, A., 
\emph{Matrix Krylov subspace methods for linear systems with multiple right-hand sides},
Numerical Algorithms, 40(2), pp. 137-156, 2005. 

\bibitem{Charnyi2017289}
Sergey Charnyi and Timo Heister and Maxim A. Olshanskii and Leo G. Rebholz
\emph{On conservation laws of Navier-Stokes Galerkin discretizations},
J. Comp. Phys., 337,  pp. 289-308, 2017.

\bibitem{AQ92}
D.N. Arnold and J. Qin,
\emph{Quadratic Velocity/Linear Pressure Stokes Elements},
Advances in Computer Methods for Partial Differential Equations VII, 
IMACS, pp. 28-34, 1992.

\bibitem{Takhirov2018}
Takhirov, Aziz and Lozovskiy, Alexander,
\emph{Computationally efficient modular nonlinear filter stabilization for high Reynolds number flows},
Adv. Comp. Math., 44(1), pp. 295-325, 2017.
\end{thebibliography}


\end{document}